\documentclass[10pt]{amsart}

\usepackage{amstext}
\usepackage{amsbsy}
\usepackage{amsopn}
\usepackage{amsxtra}

\usepackage[utf8]{inputenc}
\usepackage[T1]{fontenc}
\usepackage{graphicx}
\usepackage{amsmath,amsfonts,amsthm,amssymb}
\usepackage{enumerate}
\usepackage{csquotes}
\usepackage{hyperref}
\usepackage{cleveref}
\usepackage{xspace}
\usepackage{xcolor}
\newtheorem{thm}{Theorem}[section]
\newtheorem{lem}[thm]{Lemma}
\newtheorem{defn}[thm]{Definition}
\newtheorem{prop}[thm]{Proposition}
\newtheorem{exam}[thm]{Example}

\newtheorem{cor}[thm]{Corollary}
\newtheorem{rem}[thm]{Remark}

\DeclareMathOperator{\Aut}{Aut}
\DeclareMathOperator{\Per}{Per}
\DeclareMathOperator{\Orb}{Orb}
\DeclareMathOperator{\Fix}{Fix}
\DeclareMathOperator{\Inert}{Inert}
\DeclareMathOperator{\Simp}{Simp}
\DeclareMathOperator{\Sym}{Sym}
\DeclareMathOperator{\Alt}{Alt}
\DeclareMathOperator{\MPerOp}{Per}
\DeclareMathOperator{\roots}{root}
\DeclareMathOperator{\defect}{Def}
\DeclareMathOperator{\per}{per}
\newcommand{\setsep}{\;|\;}
\newcommand{\MPer}{\MPerOp^{\min}}

\newcommand{\bbN}{\mathbb{N}}
\newcommand{\bbZ}{\mathbb{Z}}
\newcommand{\cA}{\mathcal{A}}

\newcommand{\spatial}{\hat{\Psi}}
\newcommand{\stp}{\textrm{St}_{p}}
\newcommand{\sub}{\textrm{Sub}}
\newcommand{\topt}{\uparrow}
\newcommand{\bott}{\downarrow}
\newcommand{\abs}[1]{\left\vert#1\right\vert}

\DeclareMathOperator{\aut}{Aut}
\newcommand{\autinf}{\aut^{\infty}}
\newcommand{\autinfn}{\autinf(\sigma_n)}
\newcommand{\autautinfn}{\aut(\autinfn)}
\newcommand{\autn}{\aut(\sigma_n)}
\newcommand{\autautn}{\aut(\autn)}
\DeclareMathOperator{\inert}{Inert}
\newcommand{\inertinf}{\inert^{\infty}}
\newcommand{\inertinfn}{\inertinf(\sigma_n)}

\newcommand{\orpre}{\aut^+}
\newcommand{\orpren}{\orpre(\autinfn)}

\newcommand{\degonen}{\aut_1(\autinfn)}

\newcommand{\linf}{{}^\infty}
\newcommand{\rinf}{{}^\infty}
\newcommand{\simp}{\textrm{Simp}}
\newcommand{\asa}[1]{\aut(\autinf(#1))}
\newcommand{\asan}{\asa{\sigma_n}}
\newcommand{\profint}{\hat{\bbZ}}
\newcommand{\profact}{\mathcal{N}}
\newcommand{\modulo}{\;\textrm{mod}\;}
\newcommand{\zmod}[1]{\bbZ / #1\bbZ}

\newcommand{\verraum}{Verräumlichung\xspace}
\newcommand{\smallcr}{\scriptscriptstyle \mathcal{CR}}
\newcommand{\sinfn}{\mathcal{S}^{\infty}(\sigma_{n})}
\newcommand{\crinfn}{\mathcal{CR}^{\infty}(\sigma_{n})}

\DeclareMathOperator{\id}{id}

\begin{document}
\title[Hyperspatiality and stabilized automorphism groups]{Hyperspatiality for isomorphisms of stabilized automorphism groups of shifts of finite type}
\author{Jeremias Epperlein}
\author{Scott Schmieding}
\thanks{SS was partially supported by the National Science Foundation grant DMS-2247553.}
\maketitle
\begin{abstract}
Given a homeomorphism $T \colon X \to X$ of a compact metric space $X$, the stabilized automorphism group $\aut^{\infty}(T)$ of the system $(X,T)$ is the group of self-homeomorphisms of $X$ which commute with some power of $T$. We study the question of spatiality for stabilized automorphism groups of shifts of finite type. We prove that any isomorphism $\Psi \colon \aut^{\infty}(\sigma_{m}) \to \autinfn$ between stabilized automorphism groups of full shifts is spatially induced by a homeomorphism $\hat{\Psi}$ between respective stabilized spaces of chain recurrent subshifts. This spatialization in particular gives a bijection between the sets of periodic points which intertwines some powers of the shifts, and this bijection recovers the isomorphism at the level of the faithful actions on the sets of periodic points. We also prove that the outer automorphism group of $\autinfn$ is uncountable, and deduce several other properties of $\autinfn$ using the spatiality results.
\end{abstract}

\section{Introduction}
For a homeomorphism $T \colon X \to X$ of some compact metric space, the automorphism group of the system $(X,T)$ is the group $\aut(T)$ of all homeomorphisms $\varphi \colon X \to X$ such that $\varphi \circ T = T \circ \varphi$. A broad question that has been of significant interest is the following: to what extent does the structure of the automorphism group reflect the dynamics of the underlying system? Automorphism groups of dynamical systems have been heavily studied especially in the setting of symbolic systems. There the most fundamental objects are the full shift systems $(X_{n},\sigma_{n})$ where $X_{n}$ is the compact space of biinfinite strings on $n$ symbols and $\sigma_{n} \colon X_{n} \to X_{n}$ is the shift map $\sigma_{n}(x)_{i} = x_{i+1}$. The groups $\aut(\sigma_{n})$ were first studied by Hedlund et al.~\cite{Hedlund1969}, and while much has been written on them, they remain quite mysterious (see both~\cite{BLR} and the very brief survey in~\cite[Section 7]{BSstablealgebra}).

Recently, a stabilized version of the automorphism group of a system has been studied in the context of symbolic systems. The stabilized automorphism group of a system $(X,T)$ is the group $\autinf(T)$ of all homeomorphisms $\varphi \colon X \to X$ such that $\varphi \circ T^{k} = T^{k} \circ \varphi$ for some $k \in \mathbb{N}$. It is clear that the stabilized automorphism group always contains the nonstabilized one, and in general, the difference between the two depends on the underlying dynamics (see for instance~\cite{Jones2022,EspinozaJones2024} for an analysis of some stabilized groups outside of the shift of finite type setting). Some fundamental properties of these groups in the setting of shifts of finite type can be found in~\cite{hartmanStabilizedAutomorphismGroup2022}.

The stabilized automorphism groups of (nontrivial) shifts of finite type are always countable, and have a rich, complicated group structure. In~\cite{hartmanStabilizedAutomorphismGroup2022}, it was shown that the stabilized automorphism group of a full shift is always an extension of a finite rank free abelian group by an infinite simple group, called the subgroup of stabilized inert automorphisms. This was generalized by Salo in~\cite{SaloGateLattices}, where it was proved that for a mixing shift of finite type (and for many classes of shifts of finite type over more general groups), the subgroup of stabilized inert\footnote{For subshifts over more general groups, by the subgroup of stabilized inert automorphisms we mean the commutator of the gate lattice subgroup, as defined in~\cite{SaloGateLattices}.} automorphisms is always an infinite simple group. It was shown in~\cite{hartmanStabilizedAutomorphismGroup2022} that if $\aut^{\infty}(\sigma_{m})$ and $\aut^{\infty}(\sigma_{n})$ are isomorphic, then $m$ and $n$ must have the same number of prime factors. Following this, it was then proved in~\cite{schmiedingLocalMathcalEntropy2022} that $\autinf(\sigma_{m})$ and $\autinf(\sigma_{n})$ are isomorphic if and only if $m^{k} = n^{j}$.




Here we concern ourselves with the question of \emph{spatiality} for isomorphisms of stabilized automorphism groups of shifts of finite type: if $\autinf(\sigma_{X})$ and $\autinf(\sigma_{Y})$ are isomorphic for some shifts of finite type $(X,\sigma_{X})$ and $(Y,\sigma_{Y})$, to what extent must the isomorphism be spatial? By spatiality we mean the following. Suppose $G_{1},G_{2}$ are groups with prescribed faithful actions $\rho_{i} \colon G_{i} \to \textrm{Homeo}(X_{i})$ where $X_{i}$ are some compact metric spaces. We say an isomorphism $\Psi \colon G_{1} \to G_{2}$ is \emph{spatial} if there exists a homeomorphism $\psi \colon X_{1} \to X_{2}$ such that
$\rho_2(\Psi(g)) = \psi \circ \rho_1(g) \circ \psi^{-1}$ for all $g \in G_{1}$

Spatiality of isomorphisms for various classes of groups has a long history, which we do not attempt to comprehensively survey here. Giordano, Putnam, and Skau proved in~\cite{GPSfullgroups} that any isomorphism of topological full groups of Cantor minimal systems must be spatial. Much more generally, an important collection of spatiality results was proved by Rubin in~\cite{Rubin1989,RubinShortProof} for sufficiently nice actions on locally compact spaces. Rubin's results have found many applications to a wide range of groups (see, for instance, the list of references found in~\cite{RubinShortProof}). We note that for both the topological full group case, and so-called Rubin actions, a key property is the existence of many elements with nontrivial support.

In contrast to settings such as topological full groups or Rubin actions, there are various difficulties which arise for the question of spatiality in the context of automorphism groups of shifts of finite type. Arguably the most substantial is that there are \emph{no elements with nontrivial support}: if $\alpha \in \autinf(\sigma_{n})$ acts by the identity on a nonempty open set, then $\alpha$ must be the identity map. In fact, this phenomenon holds once one considers the automorphism group of any transitive subshift, due to the presence of a dense orbit under the shift map. As a consequence, many well-known techniques do not apply, and we must use different methods. Here we take an approach to spatiality by working with the space of all subsystems of our underlying system.

We briefly outline some notation in order to state our results. Given a compact metric space $X$, denote by $\mathcal{K}(X)$ the space of all nonempty compact subsets of $X$. With the Hausdorff metric, $\mathcal{K}(X)$ is a compact metric space. If $\rho \colon G \to \textrm{Homeo}(X)$ is an action of a group $G$ on $X$, there is an induced action $\rho^{\mathcal{K}} \colon G \to \textrm{Homeo}(\mathcal{K}(X))$.


\begin{defn}
Suppose $G_{1},G_{2}$ are groups with actions $\rho_{i} \colon G_{i} \to \textrm{Homeo}(X_{i})$ where $X_{i}$ are some compact metric spaces. We say an isomorphism $\Phi \colon G_{1} \to G_{2}$ is \emph{hyperspatial} if $\Psi$ is spatial with respect to the actions $\rho_{i}^{\mathcal{K}} \colon G \to \textrm{Homeo}(\mathcal{K}(X_{i}))$, i.e. there exists a homeomorphism $\psi \colon \mathcal{K}(X_{1}) \to \mathcal{K}(X_{2})$ such that for all $g \in G_{1}$, we have
\[\rho_{2}^{\mathcal{K}}(\Phi(g)) = \psi \circ \rho_{1}^{\mathcal{K}}(g) \circ \psi^{-1}.\]
\sloppy If in addition the actions of $G_{1}$ and $G_{2}$ leave invariant some subsets $U_{1},U_{2}$ of $\mathcal{K}(X_{1}), \mathcal{K}(X_{2})$, respectively, we say an isomorphism $\Phi \colon G_{1} \to G_{2}$ is \emph{$U_{1},U_{2}$-hyperspatial} if it is spatial with respect to the actions $\rho_{i}^{\mathcal{K}} \colon G_{i} \to \textrm{Homeo}(U_{i})$.
\end{defn}

For a system $(X,T)$, we define the \emph{space of subsystems} of $(X,T)$ to be the subset $\mathcal{S}(T) \subseteq \mathcal{K}(X)$ consisting of compact subsets $Y \subseteq X$ for which $T(Y)=Y$. We define the \emph{stabilized space of subsystems} to be the subset $\mathcal{S}^{\infty}(T) \subseteq \mathcal{K}(X)$ which consists of compact subsets $Y \subseteq X$ for which $T^{k}(Y) = Y$ for some $k \in \mathbb{N}$.

We specialize now to the case of full shifts $(X_{n},\sigma_{n})$ and the groups $\aut(\sigma_{n})$ and $\autinfn$. Since every element of $\aut(\sigma_{n})$ commutes with $\sigma_{n}$, the induced action of $\aut(\sigma_{n})$ on $\mathcal{K}(X_{n})$ leaves $\mathcal{S}(\sigma_{n})$ invariant, and hence induces in a natural way an action of $\aut(\sigma_{n})$ on $\mathcal{S}(\sigma_{n})$. Likewise, the action of $\autinf(\sigma_{n})$ on $\mathcal{K}(X_{n})$ leaves invariant $\mathcal{S}^{\infty}(\sigma_{n})$, giving rise to an action of $\autinf(\sigma_{n})$ on the stabilized space of subsystems of $(X_{n},\sigma_{n})$. We note that, while the action of $\aut(\sigma_{n})$ on $\mathcal{S}(\sigma_{n})$ has nontrivial elements which act by the identity on open subsets, for many clopen subsets of $\mathcal{S}(\sigma_{n})$ there are no elements with support contained within that clopen subset.

Recall a system $(X,T)$ is said to be chain recurrent if for all $\epsilon > 0$ there exists a finite sequence of points $x = x_{0},\ldots,x_{n}=x$ such that $d(T(x_{i}),x_{i+1}) < \epsilon$ for all $0 \le i < n$. The space of chain recurrent subsystems of $(X_{n},\sigma_{n})$ is a large compact subset of $\mathcal{S}(\sigma_{n})$; it contains, for instance, all finite disjoint unions of transitive subshifts in $X_{n}$. We prove in Proposition~\ref{prop:CRlimitperiodic} that the set of chain recurrent subshifts in $\mathcal{S}(\sigma_{n})$ is precisely the closure of the collection of finite subsystems in $\mathcal{S}(\sigma_{n})$.

For each $k \ge 1$, the space of chain recurrent $\sigma_{n}^{k}$-subshifts $\mathcal{CR}(\sigma_{n}^{k})$ in $\mathcal{S}(\sigma_{n}^{k})$ is invariant under the action of $\aut(\sigma_{n}^{k})$. We define the stabilized space of chain recurrent subshifts $\mathcal{CR}^{\infty}(\sigma_{n})$ to be the union $\bigcup_{k=1}^{\infty}\mathcal{CR}(\sigma_{n}^{k})$. It is immediate from the definition that the space $\mathcal{CR}^{\infty}(\sigma_{n})$ is invariant under the action of $\autinf(\sigma_{n})$. Given $\alpha \in \autinf(\sigma_{n})$, we let $\alpha_{\scriptscriptstyle \mathcal{CR}}$ denote its corresponding action on $\mathcal{CR}^{\infty}(\sigma_{n})$.

The space $\mathcal{CR}^{\infty}(\sigma_{n})$ inherits a topology as a subspace of $\mathcal{K}(\sigma_{n})$, but also carries the final topology on $\mathcal{CR}^{\infty}(\sigma_{n})$ coming from the filtration $\mathcal{CR}^{\infty}(\sigma_{n}) = \bigcup_{k=1}^{\infty}\mathcal{CR}(\sigma_{n}^{k})$. In the final topology, a set $U$ is open in $\mathcal{CR}^{\infty}(\sigma_{n})$ if and only if $U \cap \mathcal{CR}(\sigma_{n}^{k})$ is open for every $k$. Since each $\aut(\sigma_{n}^{k})$ acts continuously on $\mathcal{CR}(\sigma_{n}^{k})$, the action of $\autinf(\sigma_{n})$ is also continuous with respect to this final topology.

Our main theorem is the following hyperspatiality result.


\begin{thm}[Theorem~\ref{thm:maintheorem2} in the main text]\label{thm:maintheorem}
Let $m,n$ be positive integers and suppose that $\Psi \colon \autinf(\sigma_{m}) \to \autinf(\sigma_{n})$ is an isomorphism of groups. Then there exists a homeomorphism $\hat{\Psi} \colon \mathcal{CR}^{\infty}(\sigma_{m}) \to \mathcal{CR}^{\infty}(\sigma_{n})$ with respect to the final topologies such that the isomorphism $\Psi$ is $\mathcal{CR}^{\infty}(\sigma_{m}), \mathcal{CR}^{\infty}(\sigma_{n})$-hyperspatial via $\hat{\Psi}$, i.e. for all $\varphi \in \autinf(\sigma_{m})$ we have
$$\Psi(\varphi)_{\scriptscriptstyle \mathcal{CR}} = \hat{\Psi}\varphi_{\scriptscriptstyle \mathcal{CR}}\hat{\Psi}^{-1}.$$
Moreover, the map $\hat{\Psi}$ takes $\sigma_{m}$-periodic points to $\sigma_{n}$-periodic points, and intertwines the actions of some powers of $\sigma_{m}$ and $\sigma_{n}$ on their respective set of periodic points.
\end{thm}

The theorem is interesting even at just the level of periodic points. As noted in the statement, Theorem~\ref{thm:maintheorem} in particular implies such an isomorphism $\Psi$ is spatial at the level of periodic points, and this spatial map on the periodic points intertwines some powers of the respective shift maps. This has a variety of implications; for instance, we use it to conclude (Corollary~\ref{cor:freeness}) that freeness is a characteristic property for the action of $\autinfn$ on $X_{n}$.

The heart of the paper is the construction, given an isomorphism $\Psi$ between stabilized automorphism groups, of the map $\hat{\Psi}$ from Theorem~\ref{thm:maintheorem}. We call this map $\hat{\Psi}$ the \verraum\footnote{German
for \enquote{spatialization}.} of the isomorphism $\Psi$. Initially $\hat{\Psi}$ is built as a map on periodic points, where we prove it spatially implements $\Psi$ on the stabilized inert automorphisms. We then show that it spatially implements $\Psi$ on all of the stabilized automorphism group. Continuity and the extension to the space of chain recurrent subshifts is obtained by passing through the space of subgroups of $\autinf(\sigma_{n})$.

We also show that the group of outer automorphisms of $\autinfn$ is always uncountable. More precisely, we prove the following.
\begin{thm}\label{thm:exactsequenceaut1}
Let $n \ge 2$ and let $\aut_{1}(\autinfn)$ denote the group of automorphisms of $\autinfn$ which take $\sigma_{n}$ to $\sigma_{n}^{\pm 1}$. Then there is an exact sequence
$$1 \longrightarrow \hat{\mathbb{Z}} \stackrel{\mathcal{N}}\longrightarrow \aut_{1}(\autinfn) \stackrel{\mathcal{V}}\longrightarrow \textnormal{Homeo}(\mathcal{CR}^{\infty}(\sigma_{n}))$$
where $\hat{\mathbb{Z}}$ denotes the profinite integers and $\mathcal{V}$ is the \verraum construction. Moreover, $\mathcal{N}(a)$ is inner if and only if $a=(a_{k})_{k \in \bbN}$ belongs to the canonical copy of $\mathbb{Z}$ in $\hat{\mathbb{Z}}$.
\end{thm}

We show in Theorem~\ref{thm:left-cosets-countable} that the number of left cosets of $\mathcal{N}(\hat{\bbZ})$ in $\aut(\autinfn)$ is always countable.

Given $\Psi \in \aut(\autinfn)$, the \verraum $\hat{\Psi}$ is well-defined on all periodic points, which forms a dense subset of $X_{n}$. It is natural to ask whether $\hat{\Psi}$ might be uniformly continuous in the usual topology on $X_{n}$, in which case it would extend to a unique homeomorphism $\hat{\Psi} \colon X_{n} \to X_{n}$ giving spatiality of $\Psi$ at the level of $X_{n}$. In general, this can not hold for arbitrary $\Psi$, since the $\mathcal{N}$-image of nonintegral elements in $\hat{\mathbb{Z}}$ are not induced by any self-homeomorphism of $X_{n}$.
But we do not know whether this copy of $\hat{\mathbb{Z}}$ in $\aut_{1}(\autinfn)$ is in fact the only obstruction to spatiality at the level of $X_{n}$. In other words, it is possible that if $\Psi \colon \autinfn \to \autinfn$ is an isomorphism for which $\Psi(\sigma_{n}) = \sigma_{n}^{\pm 1}$, then up to composing with an element in the image of $\mathcal{N}$, the $\mathcal{CR}^{\infty}(\sigma_{n})$-spatiality of the \verraum $\hat{\Psi}$ given by Theorem~\ref{thm:maintheorem} can be upgraded to a homeomorphism $\hat{\Psi} \colon X_{n} \to X_{n}$. To be more precise, we pose the following question.\\

\textbf{Question: } Let $n$ be a positive integer and let $\aut_{1}(\autinfn)$ denote the group of automorphisms of $\autinfn$ which take $\sigma_{n}$ to $\sigma_{n}^{\pm 1}$. Is it true that for every $\Psi \in \aut_{1}(\autinfn)$, there exists $a \in \hat{\mathbb{Z}}$ such that $\mathcal{N}(a) \circ\Psi$ is induced by a homeomorphism $\hat{\Psi} \colon X_{n} \to X_{n}$?\\

A positive answer to the question is equivalent to showing that, modulo the image of $\mathcal{N}$, the \verraum $\hat{\Psi}$ is uniformly continuous on periodic points in the product topology on $X_{n}$. A consequence would be that $\aut_{1}(\autinfn)/\textrm{Image}(\mathcal{N})$ is isomorphic to a subgroup of $\textrm{Homeo}(X_{n})$.

\section{Background}
\sloppy A system $(X,T)$ means a homeomorphism $T \colon X \to X$ where $X$ is a compact metric space. Denote by $X_{n}$ the space $\{0,\ldots,n-1\}^{\mathbb{Z}}$ with the product topology. We write elements in $X_{n}$ as biinfinite sequences $(x_{n})_{n \in \mathbb{Z}}$. The shift map $\sigma_{n} \colon X_{n} \to X_{n}$ defined by $\left(\sigma_{n}(x)\right)_{n} = x_{n+1}$ defines a homeomorphism, and we call the system $(X_{n},\sigma_{n})$ the full shift on $n$ symbols. By a subshift we mean a system $(X,\sigma_{X})$ where $X$ is a compact $\sigma_{n}$-invariant subset of $X_{n}$ for some positive integer $n$ and $\sigma_{X}$ denotes the restriction of $\sigma_{n}$ to $X$.

For a system $(X,T)$ we define periodic point sets
\begin{align*}
    \Per_{k}(T) &= \{x \in X \mid T^{k}(x) = x\},\\
    \Per_{k}^{\mathrm{min}}(T) &= \{x \in X \mid T^{k}(x) = x, T^{j}(x) \ne x \textrm{ for all } 0 < j < k\},\\
\Per(T) &= \bigcup_{k=1}^{\infty}\Per_{k}(T).
\end{align*}
For a periodic point $x$ denote by $\per(x)$ the minimal
period of $x$, i.e. \[\per(x) := \min \{k \in \bbN \setsep x \in \Per_k(T)\}.\]

The group of automorphisms of a system $(X,T)$ is defined by
$$\aut(T) = \{\varphi \colon X \to X \mid \varphi \circ T = T \circ \varphi\}$$
where the group operation is composition. The stabilized group of automorphisms is defined by
$$\aut^{\infty}(T) = \{\varphi \colon X \to X \mid \varphi \circ T^{k} = T^{k} \circ \varphi \textrm{ for some } k\}.$$
Clearly $\aut(T) \subseteq \aut^{\infty}(T)$ holds for any system.

Our interest here will be the groups $\autinfn$, for which we recount some known facts now. Let $n \ge 2$. The group $\autinfn$ is not finitely generated, has trivial center, and is not residually finite~\cite{hartmanStabilizedAutomorphismGroup2022}. For any $k \ge 1$ the systems $(X_{n},\sigma_{n}^{k})$ and $(X_{n^{k}},\sigma_{n^{k}})$ are topologically conjugate; since $\aut^{\infty}(\sigma_{n}^{k}) = \aut^{\infty}(\sigma_{n})$, it follows that the groups $\autinfn$ and $\aut^{\infty}(\sigma_{n^{k}})$ are isomorphic.

Associated to a shift of finite type is a dimension group, which for $(X_{n},\sigma_{n})$ is isomorphic to  $\mathbb{Z}[\tfrac{1}{n}]$, the additive subgroup of rationals whose denominators are powers of $n$. There is a dimension representation homomorphism
$$\pi_{n} \colon \aut(\sigma_{n}) \to \aut_{+}(\mathbb{Z}[\tfrac{1}{n}])$$
where $\aut_{+}(\mathbb{Z}[\tfrac{1}{n}])$ denotes the group of automorphisms of $\mathbb{Z}[\tfrac{1}{n}]$ as an abelian group which preserve the nonnegative elements. The dimension representation extends to a stabilized dimension representation
$$\pi_{n}^{\infty} \colon \autinfn \to \aut_{+}(\mathbb{Z}[\tfrac{1}{n}]).$$
The following proposition is an immediate consequence of the results in~\cite{boyleAutomorphismGroupShift1988}; a proof in the stabilized language can be found in~\cite{hartmanStabilizedAutomorphismGroup2022}.
\begin{prop}
The automorphism group $\aut_{+}(\mathbb{Z}[\tfrac{1}{n}])$ of the dimension group of $(X_n,\sigma_n)$ is isomorphic to $\mathbb{Z}^{\omega(n)}$, where $\omega(n)$ denotes the number of distinct prime divisors of $n$. Furthermore, the stabilized dimension group representation is surjective.
\end{prop}
Writing the prime decomposition $n = p_{1}^{k_{1}}\cdots p_{\ell}^{k_{\ell}}$, upon identifying $\aut_{+}(\mathbb{Z}[\tfrac{1}{n}])$ with $\mathbb{Z}^{\ell}$, a calculation shows that $\pi_{n}^{\infty}(\sigma_{n}) = (k_{1},\ldots,k_{\ell})$.

The kernel of the (stabilized) dimension representation is called the subgroup of \emph{(stabilized) inert automorphisms}, denoted by
\begin{align*}
    \inert(\sigma_{n}) &= \textrm{ker}\, \pi_{n},\\
    \inertinfn &= \textrm{ker}\, \pi_{n}^{\infty}.
\end{align*}
For $k \ge 1$ we define $\textrm{Inert}(\sigma_{n}^{k}) = \inertinfn \cap \aut(\sigma_{n}^{k})$. The following was proved in~\cite{hartmanStabilizedAutomorphismGroup2022}; a simpler, and more general, proof can be found in~\cite{SaloGateLattices}.
\begin{thm}[\cite{hartmanStabilizedAutomorphismGroup2022}]
Let $n \ge 2$. The subgroup $\inertinfn$ is equal to the commutator subgroup of $\autinfn$, and is an infinite simple group. Consequently, the stabilized inert subgroup $\inertinfn$ is a characteristic subgroup of $\autinfn$.
\end{thm}

\subsection{Stabilized simple automorphisms}
Fix $n \ge 2$ and let $\Gamma_{n}$ be a directed graph with one vertex and $n$ edges labeled $\{0,\ldots,n-1\}$. We identify $X_{n}$ with the set of biinfinite walks on $\Gamma_{n}$. Given $k \ge 1$, we let $\Gamma_{n}^{(k)}$ denote the directed graph with one vertex and labeled directed edges obtained from length $k$ walks on $\Gamma_{n}$. A graph automorphism $\tau$ of $\Gamma_{n}^{(k)}$ induces an automorphism $\tilde{\gamma} \in \aut(\sigma_{n}^{k})$ via
$$\tilde{\gamma}(\ldots x_{\scriptscriptstyle -1}x_{\scriptscriptstyle 0}x_{\scriptscriptstyle 1} \ldots) = \ldots \gamma(x_{\scriptscriptstyle -k}\ldots x_{\scriptscriptstyle -1})\gamma(x_{\scriptscriptstyle 0}\ldots x_{\scriptscriptstyle k-1})\gamma(x_{\scriptscriptstyle k} \ldots x_{\scriptscriptstyle 2k-1})\ldots.$$

Letting $\Sym(\Gamma_{n}^{(k)})$ denote the group of automorphisms of $\Gamma_{n}^{(k)}$, we have injective homomorphisms
$$\Sym(\Gamma_{n}^{(k)}) \to \aut(\sigma_{n}^{k}).$$
We denote by $\simp^{(k)}(\Gamma_{n})$ the image of this map, and by $\simp^{(k)}_{ev}(\Gamma_{n})$ the image of the alternating subgroup $\textrm{Alt}(\Gamma_{n}^{(k)})$ in $\aut(\sigma_{n}^{k})$. It is easy to see that for every $k \ge 1$, the group $\simp^{(k)}(\Gamma_{n})$ is isomorphic to the symmetric group on $n^{k}$ symbols. For every $k,l \ge 1$ there are containments
$$\simp^{(k)}(\Gamma_{n}) \subseteq \simp^{(kl)}(\Gamma_{n})$$
and we define
$$\simp^{(\infty)}(\Gamma_{n}) = \bigcup_{k=1}^{\infty}\simp^{(k)}(\Gamma_{n}) \subseteq \autinfn.$$

The subgroup $\simp^{(\infty)}(\sigma_{n})$ is contained in the stabilized inert subgroup $\inertinfn$.

A key lemma which will be used again and again is the following.
The result originally appeared in
\cite[Theorem on p. 970]{boyleEventualExtensionsFinite1988} see also
\cite[Theorem 3.3 ]{kariRepresentationReversibleCellular1996}
and \cite[Lemma 5.6]{hartmanStabilizedAutomorphismGroup2022}
\begin{lem}[Boyle-Kari representation Lemma]
  \label{lem:boyle-rep}
Let $n,k \ge 2$, let $\gamma \in \Inert(\sigma_n^{k})$ and $\ell \in \bbN$.
  For every $t\geq \ell$
  there are automorphisms $\tau_1,\tau_2 \in \Simp^{(2kt)}(\Gamma_n)$
  such that
  $\gamma =\sigma^{-kt} \circ \tau_1 \circ \sigma^{kt} \circ \tau_2$.
\end{lem}

A consequence of Lemma~\ref{lem:boyle-rep} is that the $\sigma_{n}$-conjugates of $\simp^{(\infty)}(\Gamma_{n})$ generate $\inert^{\infty}(\sigma_{n})$.



\subsection{Centralizers}
We prove here some results on centralizers that will be of use later.
\begin{thm}[{Shifting implies uniform shifting,
    \cite[Theorem 2.5]{boylePeriodicPointsAutomorphisms1987}}]
  \label{thm:shifting-has-to-be-uniform}
Let $\varphi \in \Aut(\sigma_n)$ be such that for all sufficiently
  large $k \in \bbN$, if $x \in \Per_{k}(\sigma_{n})$ then $\varphi(x)$ lies in the same $\sigma_{n}$-orbit as $x$. Then $\varphi = \sigma_n^\ell$ for some $\ell \in \bbZ$.
\end{thm}


\begin{thm}[{{$\autn$ acts pretty transitively on $\Per_k(\sigma_n)$},
  \cite[Theorem 3.6]{boylePeriodicPointsAutomorphisms1987}}]
  \label{lem:pretty-trans}
  Let $k \in \bbN$ and let $x,y \in \Per^{\mathrm{min}}_k(\sigma_n)$ be two points
  with least period $k$ not in the same $\sigma_n$
  orbit. Then there is an involution $\alpha \in \autn$
  which exchanges $x$ and $y$ and fixes all points
  $\Per_k(\sigma_n)$ which do not lie in the $\sigma_n$
  orbits of $x$ and $y$.
\end{thm}




For a set $A$, we let $\Sym(A)$ denote the group of permutations of $A$.

\begin{thm}
  \label{thm:centralizer-aut-in-per}
The centralizer of the image of $\autn$ in $\Sym(\Per(\sigma_{n}))$ is contained in the set of bijections which preserve $\sigma_{n}$-orbits for $n \ge 3$.  If $n=2$, then there is one additional
  generator given by the involution exchanging exactly the two fixed points.
\end{thm}
\begin{proof}
  Supposes that $\varphi$ is a map in the centralizer of the image of $\autn$.
  Since $\sigma_n \in \autn$, we have $\varphi \circ \sigma_{n} = \sigma_{n} \circ \varphi$, so $\varphi$ induces
  a permutation of $\sigma_n$ orbits. Let $\Orb_{k}(\sigma_{n})$ denote the set of $\sigma_{n}$-orbits of length $k$ and consider the homomorphism $\autn \to \Sym(\Orb_k(\sigma_n))$ defined by sending an automorphism to the permutation of $\Orb_{k}(\sigma_{n})$ that it induces. From \Cref{lem:pretty-trans} we know that
  this homomorphism is
  surjective. Hence the permutation of $\Orb_{k}(\sigma_{n})$ induced by $\varphi$ must
  be in the center of $\Sym(\Orb_k(\sigma_n))$
  which is trivial as long as $k\geq 2$ or $n \geq 3$. Finally, if $n=2$, it is easy to check that the involution of $\Sym(\Per(\sigma_{n}))$ defined by exchanging the two $\sigma_{n}$-fixed points and leaving every other point fixed is in the centralizer of the image of $\autn$ in $\Sym(\Per(\sigma_{n}))$.
\end{proof}

\begin{thm}\label{cor:centralizer-in-per}
  The centralizer of the image of $\inertinfn$ in $\Sym(\Per(\sigma_n))$ is trivial.
\end{thm}
\begin{proof}
  Let $\varphi \in \Sym(\Per(\sigma_n))$.
  Assume there is a point $x \in \Per(\sigma_n)$
  with $\varphi(x)\neq x$.
  Let $y$ be a point in $\Per(\sigma_n)\setminus \{x,\varphi(x)\}$.
  Choose $k$ such that $x,y,\varphi(x)$ are all fixed by $\sigma_{n}^{k}$.
  Now there is an involution $\alpha \in \Simp^{(k)}(\Gamma_n)
  \subseteq \inertinfn$
  swapping $x$ and $y$ which fixes $\varphi(x)$.
  Then $\alpha(\varphi(x))=\varphi(x)$
  but $\varphi(\alpha(x))=\varphi(y)$, hence
  $\alpha \circ \varphi \neq \varphi \circ \alpha$.
  This contradicts the fact that $\varphi$ is in the centralizer of
  the image of
  $\inertinfn$ in $\Sym(\Per(\sigma_n))$.
\end{proof}

\begin{cor}
  \label{cor:stabilized-inert-ryan}
  The centralizer of $\inertinfn$ in $\autinfn$ is trivial.
\end{cor}
\begin{proof}
  The map $\autinfn \to \Sym(\Per(\sigma_n))$ is injective
  as $\sigma_n$-periodic points are dense.
  The result then follows directly from the previous theorem.
\end{proof}


\section{$\aut(\autinfn)$ and an embedding of $\profint$}
\label{sec:an-action-profint}
We fix $n \ge 2$ throughout this section. Denote by $\aut(\autinfn)$ the group of automorphisms of the group $\autinfn$. Let $\profint = \varprojlim \zmod{m}$ be the profinite integers.
In concrete terms, the elements of $\profint$
are sequences $(a_m)_{m \in \bbN} \in \prod_{m \in \bbN} \bbZ/ m \bbZ$
for which $a_m \equiv a_{km} \modulo m$ for all $m,k \in \bbN$.
\begin{prop}
There is an injective homomorphism
$$\mathcal{N} \colon \hat{\mathbb{Z}} \to \asan.$$
\end{prop}
\begin{proof}
We define the map
\begin{align*}
  \profact((a_m)_{m \in \bbN})(\varphi) := \sigma_n^{-a_k} \circ
  \varphi \circ \sigma_n^{a_k} \text{ for }
  \varphi \in \aut(\sigma_n^{k}).
\end{align*}
To see that $\mathcal{N}(a)$ for $a=(a_m)_{m \in \bbN}$ is well-defined, suppose $\varphi \in \aut(\sigma_{n}^{k})$, $j \ge 1$, and write $a_{jk} = a_{k} + rk$. Then considering $\varphi \in \aut(\sigma_{n}^{jk})$ we have
\begin{align*}
\mathcal{N}(a)(\varphi) &= \sigma_{n}^{-a_{jk}} \circ \varphi \circ \sigma_{n}^{a_{jk}} = \sigma_{n}^{-a_{k} - rk} \circ \varphi \circ \sigma_{n}^{a_{k}+rk}\\
&= \sigma_{n}^{-a_{k}}\circ \varphi \circ \sigma_{n}^{a_{k}}
\end{align*}
and this last expression is $\mathcal{N}(a)(\varphi)$ when considering $\varphi \in \aut(\sigma_{n}^{k})$. That $\mathcal{N}(a)$ defines an automorphism of $\autinfn$ is straightforward to check. Injectivity follows from the fact that if $a_{k} \ne 0$ for some $k$ then conjugation by $\sigma_{n}^{a_{k}}$ on $\aut(\sigma_{n}^{k})$ is nontrivial.
\qedhere
\end{proof}
Note that for any $(a_{m}) \in \hat{\bbZ}$, the group automorphism $\profact((a_m)_{m \in \bbN})$ always fixes every element in $\aut(\sigma_{n})$; in particular, it fixes $\sigma_{n}$.
\begin{rem}
  \label{rem:no-other-profinite-action}
The above construction does not work with any other element of $\autinfn$ besides powers of $\sigma$. Indeed, suppose
  $\Psi(\alpha) = \gamma_{1}^{-1} \circ \alpha \circ \gamma_{1}$ on
  $\Aut(\sigma_n)$ and
  $\Psi(\alpha) = \gamma_{2}^{-1} \circ \alpha \circ \gamma_{2}$ on
  $\Aut(\sigma_n^{k})$. Then $\gamma_{2}\circ \gamma_{1}^{-1}$ is in the
  center of $\Aut(\sigma_n)$ in $\Aut(\sigma_n^{k})$. Thus $\gamma_{2}\circ \gamma_{1}^{-1}$ commutes with $\sigma_{n}$ and must belong to $\aut(\sigma_{n})$, and hence also the center of $\aut(\sigma_{n})$. Then by Ryan's Theorem~\cite{Ryan1,Ryan2}, this implies $\gamma_{2}\circ \gamma_{1}^{-1} = \sigma_{n}^{k}$ for some $k \in \mathbb{Z}$.
\end{rem}

There is a canonical injective group homomorphism $\mathbb{Z} \to \hat{\mathbb{Z}}$ given by the map $n \mapsto (n \mod m)_{m \in \mathbb{N}}$. The map $\mathcal{N} \colon \hat{\mathbb{Z}} \to \asan$ takes this copy of $\mathbb{Z}$ in $\hat{\mathbb{Z}}$ to the automorphisms of $\autinfn$ defined by $\varphi \mapsto \sigma^{-m} \circ \varphi \circ \sigma^{m}$. In fact, these are the only spatial automorphisms in the image of $\mathcal{N}$.

\subsection{The degree of an automorphism of $\autinfn$}

In this section we fix a positive integer $n= p_1^{k_1}\dots p_\ell^{k_\ell}$ with
  $p_1<\dots<p_\ell$ primes and $k_1,\dots,k_\ell \ge 1$.


Following  \cite[Section 4.2]{schmiedingLocalMathcalEntropy2022}
we define
$$\roots(\sigma_n)=\{\gamma \in \Aut^\infty(\sigma_n) \setsep
\gamma^r = \sigma_n^s \text{ for some } r,s \neq 0\}.$$

\begin{prop}
Suppose $n$ is not a power of an integer and $\gamma \in \roots(\sigma_n)$. Then
  there are $r,t \in \bbZ \setminus \{0\}$ such that
  $\gamma^{r}=\sigma^{rt}$.
\end{prop}
\begin{proof}
This follows directly from the dimension representation. If $n$ is not a power of an integer, then $\gcd(k_1,\dots,k_n)=1$. Let
  $\gamma \in \roots(\sigma_{n})$ so
  there are $r,s \neq 0$ such that $\gamma^r=\sigma_n^s$.
  Let $\pi_{n}^\infty(\gamma) = (a_1,\dots,a_\ell) \in \bbZ^\ell$ be the dimension representation
  of $\gamma$. Then $r(a_1,\dots,a_\ell)=s(k_1,\dots,k_\ell)$
  and $r\gcd(a_1,\dots,a_\ell)=s\gcd(k_1,\dots,k_\ell)=s$.
  In particular $s=rt$ for some $t \in \bbZ \setminus\{0\}$.
\end{proof}

\begin{prop}
  For every group automorphism $\Psi$ of $\Aut^{\infty}(\sigma_n)$
  there is $k \in \bbN$ such that
  $\Psi(\sigma^k) = \sigma^{k}$ or $\Psi(\sigma^k)=\sigma^{-k}$.
\end{prop}
\begin{proof}
  First assume that $n$ is not a power of some integer.
  By
  \cite[Proposition 30]{schmiedingLocalMathcalEntropy2022}, $\Psi(\sigma_n) =:\gamma \in \roots(\sigma_n)$.
 Then by the previous proposition there are $r,t \neq 0$ such that $\gamma^{r}=\sigma_n^{rt}$, hence
  $\Psi(\sigma_n^r)=\gamma^{r}=\sigma_n^{rt}$.
  Now $(\Psi^{-1}(\sigma_n))^{rt}=\sigma_n^r$, so again
  by the dimension representation
  we see that $rt$ divides $r$ and hence $|t|=1$.

  If $n$ is a power of a integer, we can find $m,k>1$ such that
  $n=m^k$ with $m$ not a power of any integer.
  There is an isomorphism
  $\Phi: \Aut^{\infty}(\sigma_n) \to \Aut^{\infty}(\sigma_m)$
  such that $\Phi(\sigma_n)=\sigma_m^k$.
  Then $\Phi \circ \Psi \circ \Phi^{-1}$ is an automorphism of $\aut^{\infty}(\sigma_{m})$, and since $m$ is not a power of any integer, by the previous part there
  is $\ell \in \bbN$ such that
  $(\Phi \circ \Psi \circ \Phi^{-1})(\sigma_m^\ell) = \sigma_m^{\pm
    \ell}$.
  Then $\Phi^{-1}(\sigma_m^{k\ell}) = \sigma_{n}^{
    \ell}=\Psi^{-1}(\Phi^{-1}(\sigma_m^{\pm k \ell})) = \Psi^{-1}(\sigma_n^{\pm \ell})$, so $\Psi(\sigma_{n}^{\ell}) = \sigma_{n}^{\pm \ell}$.
\end{proof}

\begin{defn}
There is a minimal $d \in \bbN$ such that $\Psi(\sigma_{n}^d)=\sigma_{n}^{\pm
  d}$. We call the minimal such $d$ the \emph{degree of $\Psi$} and denote
it by $\deg(\Psi)$. If $\psi(\sigma_{n}^{\deg (\psi)})=\sigma_{n}^{\deg{\psi}}$
then we call $\psi$ \emph{orientation preserving}, otherwise we say $\Psi$ is \emph{orientation reversing}.
\end{defn}

We record some basic properties of the degree.

\begin{prop}\label{prop:degreestuff}
  Let $\Psi$ be a group automorphism of $\Aut^\infty(\sigma_n)$.
  \begin{enumerate}[(i)]
  \item $\Psi(\sigma_n^k)=\sigma_n^{\pm k}$ if and only if
    $k$ is a multiple of $\deg(\Psi)$.
    \item
    If $k$ is a multiple of $\deg(\Psi)$, then $\Psi(\aut(\sigma_{n}^{k})) = \aut(\sigma_{n}^{k})$.
  \item If $\Psi$ is an inner automorphism, which is given
  by conjugation by an element $\varphi \in \aut(\sigma_n^{k})$,
  then $\deg(\Psi) \mid k$ and $\Psi$ is orientation preserving.
\item
There is an outer automorphism $\Xi_{n}$ of $\Aut^\infty(\sigma_n)$ induced by the spatial reflection map
\begin{equation}\label{eqn:orientationreversing}
\begin{gathered}
\xi_{n} \colon X_n \to X_n\\
\xi_{n} \colon (x_k)_{\scriptscriptstyle k \in \bbZ} \mapsto (x_{-k})_{\scriptscriptstyle k \in \bbZ}.
\end{gathered}
\end{equation}
  This automorphism has degree one and is orientation reversing.
    \end{enumerate}
  \end{prop}
\begin{proof}
First consider $(i)$. That $\Psi(\sigma_{n}^{k}) = \sigma_{n}^{\pm k}$ if $k$ is a multiple of $\deg(\Psi)$ is clear. For the other direction, assume that $\Psi$ is orientation preserving; the orientation reversing case is analogous. The fixed point set $\{\varphi \in \autinfn \setsep \Psi(\varphi)=\varphi\}$ of $\Psi$ forms a subgroup of $\autinfn$, so $\{k \in \bbZ \setsep \Psi(\sigma_n^k)=\sigma_n^k\}$ is a subgroup $\bbZ$
    and hence generated by $\deg(\Psi)$.

For $(ii)$, given such a $k$, by $(i)$ we know that $\Psi(\sigma_{n}^{k}) = \sigma_{n}^{\pm k}$. Then the result follows since $\aut(\sigma_{n}^{k})$ is the centralizer of both $\sigma_{n}^{k}$ and $\sigma_{n}^{-k}$ in $\autinfn$.

Item $(iii)$ is clear from $(i)$ since conjugation by $\varphi \in \aut(\sigma_{n}^{k})$ fixes $\sigma_{n}^{k}$. That $\Xi_n$ in $(iv)$ is outer follows from $(ii)$.
\end{proof}

\begin{prop}[The subgroup of orientation preserving automorphisms]
The subgroup of orientation preserving automorphisms of
  $\autinfn$ is a normal
  subgroup of $\autinfn$ of index 2.
\end{prop}
\begin{proof}
Define a map $\tau: \autinfn \to \bbZ_2$ as follows.
  Set $\tau(\Psi)=0$
  if there is $k>0$ such that $\Psi(\sigma_n^k)=\sigma_n^k$
  and
  $\tau(\Psi)=1$
  if there is $k>0$ such that $\Psi(\sigma_n^k)=\sigma_n^{-k}$.
  It is easy to check that $\tau$ is a well-defined homomorphism
 whose kernel is the subgroup of orientation preserving automorphisms. Since the automorphism $\Xi_{n}$ defined in~\eqref{eqn:orientationreversing} of Proposition~\ref{prop:degreestuff} is in $\autautinfn$ but not orientation preserving, the conclusion follows.
\end{proof}

\begin{defn}
We denote by $\orpren$ the subgroup of orientation preserving automorphisms in $\autinfn$.
By $\Aut_m(\autinf(\sigma_n))$ we denote the subgroup
of $\autautinfn$ of all automorphisms
whose degree divides $m$.
Hence $\Aut_m(\autinfn) \subseteq \Aut_k(\autinfn)$
if $m$ divides $k$.
\end{defn}

For instance, $\Psi \in \aut_{1}(\autinfn)$ if and only if $\Psi(\sigma_{n}) = \sigma_{n}^{\pm 1}$.

\subsection{The defect}

For $k \in \bbN$ we define the following two maps
\begin{align*}
  \rho_k:  \Aut(\sigma^{k}_n) &\to \Sym(\Per_k(\sigma_n))\\
  \varphi &\mapsto \varphi_{|\Per_k(\sigma_n)}.
\end{align*}
and
\begin{align*}
  \nu_k: \Sym(\Per_k(\sigma_n)) &\to \Simp^{(k)}(\Gamma_n)\\
  \pi &\mapsto \nu_k(\pi)
\end{align*}
where $\nu_k(\pi)$ is the uniquely determined element of
$\Simp^{(k)}(\Gamma_n)$
such that $\nu_k(\pi)_{|\Per_k(\sigma_n)} = \pi$.
In other words, we have $\rho_k \circ \nu_k = \id_{\Sym(\Per_k(\sigma_n))}$
Both maps are obviously homomorphisms.
The composition in the other direction we denote by
$\cA_k=\nu_k\circ \rho_k: \Aut^{\infty}(\sigma_n) \to
\Aut^{\infty}(\sigma_n)$.

Note that by $(ii)$ of Proposition~\ref{prop:degreestuff}, for each $k$ which is divisible by $\deg(\Psi)$ we have $\Psi(\Simp^{(k)}(\Gamma_{n})) \subseteq \aut(\sigma_{n}^{k})$. For such $k$ then we may consider the intersection $\Psi(\Simp^{(k)}(\Gamma_{n})) \cap \ker \rho_{k}$.

The following proposition, while appearing minor, is in fact of fundamental importance for our construction.

\begin{prop}
Let $n \ge 2$ and $\Psi \in \aut(\autinfn)$. There exists $M \in \bbN$ such that $\Psi(\Simp^{(k)}(\Gamma_n)) \cap \ker \rho_{k}
  = \{\id\}$
  for all $k \geq  M$ with $\deg(\Psi) \mid k$.
\end{prop}
\begin{proof}
First suppose $n \ge 3$. The group  $\Psi(\Simp^{(k)}(\Gamma_n))$ is isomorphic to $\Sym(n^k)$, so
  $\Psi(\Simp^{(k)}(\Gamma_n))$ for $k\geq 3$ has only one nontrivial normal subgroup, namely
  $\Psi(\Simp_{ev}^{(k)}(\Gamma_{n}))$.
  Assume, for the sake of contradiction, that there is no such $M$. Then there is an increasing sequence of indices $k_i \geq 3$, each divisible by $\deg(\Psi)$, such that $\Psi(\Simp^{(k_i)}(\Gamma_n)) \cap \ker \rho_{k_i} \neq
  \{\id\}$.
  Then $\Psi(\Simp^{(k_i)}_{ev}(\Gamma_n)) \subseteq \ker \rho_{k_i}$ and hence $\Psi(\Simp^{(1)}_{ev}(\Gamma_n)) \subseteq \bigcap_{i} \ker \rho_{k_i}$ for all $i$. By density of periodic points, the intersection $\bigcap_{i} \ker \rho_{k_i}$ must be equal to $\{\textrm{id}\}$. But this is a contradiction, since $\Simp^{(1)}_{ev}(\Gamma_{n})$ is nontrivial for $n \geq 3$.

For $n=2$, we can slightly modify the argument to work, as follows.
First, the statement holds for even $k$ using the same argument as above together with the fact that
$\Simp_{ev}^{(2)}(\Gamma_{2})$ is non-trivial.
Now let $k \geq M$ with $\deg(\Psi)\mid k$
and consider $\varphi \in \Psi(\Simp^{(k)}(\Gamma_{n})) \subseteq  \Psi(\Simp^{(2k)}(\Gamma_{n}))$.
Since the statement holds for even $k$, $\varphi$ acts non-trivially on
$\Per_{2k}(\sigma_n)$ and since $\varphi \in \Psi(\Simp^{k}(\Gamma_{n}))$
it also acts non-trivially on $\Per_{k}(\sigma_n)$.
\end{proof}

\begin{defn}[Defective set]
  Call the set of all $k \geq 3$
  which are multiples of $\deg(\Psi)$
  and
  for which $\Psi(\Simp^{(k)}(\sigma_n)) \cap \ker \rho_{k}
  \neq \{\id\}$ the \emph{defective set} of $\Psi$. We denote this
  set by $\defect(\Psi)$.
\end{defn}

The previous proposition shows that the defective set $\defect(\Psi)$ is always finite.

\begin{rem} The defective set has the following properties:
  \begin{enumerate}[(i)]
    \item If $\Psi$ is inner and given by conjugation
      with $\psi \in \Aut(\sigma^{\ell}_n)$ then $\defect(\Psi)=\emptyset$.
      Indeed, assume otherwise that $\ell k \in \defect(\Psi)$. Then
      \[(\psi^{-1}\circ \varphi \circ \psi)_{|\Per_{\ell k}(\sigma_n)} =
      \id_{\Per_{\ell k}(\sigma_n)}\] for some nontrivial $\varphi \in \Simp^{(\ell k)}(\Gamma_n)$,
      i.e. $\varphi(\psi(x))=\psi(x)$ for $x \in \Per_{\ell k}(\sigma_n)$.

      Thus
      $\varphi_{|\Per_{\ell k}(\sigma_n)} = \varphi_{|\psi(\Per_{\ell k}(\sigma_n))}=\id_{\psi(\Per_{\ell k}(\sigma_n))}=
     \id_{\Per_{\ell
          k}(\sigma_n)}$, and hence $\varphi=\id$, a contradiction.

    \item For $\Xi_{n}$ the orientation reversing automorphism defined in Proposition~\ref{prop:degreestuff}, $\defect(\Xi_{n})=\emptyset$.
  \end{enumerate}
\end{rem}

In the next section we will show that the defective set is actually always empty.

\section{The \verraum}
We now begin construction of the \verraum associated to an automorphism $\Psi \in \aut(\autinfn)$. Fix $n \ge 2$. If $k \in \bbN$ with $\deg(\Psi) \mid k$ then $\Psi(\sigma_{n}^{k}) = \sigma_{n}^{\pm k}$, and we have the following homomorphisms:
\begin{align*}
  \Sym(\Per_k(\sigma_n)) \overset{\nu_k}{\longrightarrow}
  \Simp^{(k)}(\Gamma_n) \overset{\Psi}{\longrightarrow}
  \Aut(\sigma^{k}_n) \overset{\rho_k}{\longrightarrow}
  \Sym(\Per_k(\sigma_n)).
\end{align*}
Suppose in addition that $k \not\in \defect(\Psi)$. Then $\Psi(\Simp^{(k)}(\Gamma_n)) \cap \ker \rho_{k} = \{\textrm{id}\}$, and
the composition of these maps is an automorphism.
By assumption, $n \ge 2$, so if $k \geq 3$ then the set $\Per_{k}(\sigma_{n})$ has at least eight elements. Since every automorphism of a symmetric group on at least seven elements is inner, the automorphism $\rho_{k} \circ \Psi|_{\Simp^{(k)}(\Gamma_{n})} \circ \nu_{k}$ of $\Sym(\Per_{k}(\sigma_{N}))$ must be inner. It follows we can find a uniquely determined bijective map
$\hat{\Psi}^{(k)}: \Per_k(\sigma_n) \to \Per_k(\sigma_n)$
such that
\begin{align*}
  \rho_k(\Psi(\nu_k(\pi)))=(\hat{\Psi}^{(k)})^{-1} \circ \pi \circ \hat{\Psi}^{(k)}
\end{align*}
In this section, we will stitch these maps together
to first the build the Verräumlichung of $\Psi$ on the periodic points.

To simplify notation we
define
$$I_\Psi = \{k \setsep k\geq 3, \, k \in (\deg{\Psi}) \bbN, \, k \not\in \defect{\Psi}\}.$$

Using the fact that $\nu_k$ is actually an isomorphism,
we obtain the following theorem.
\begin{thm}[Local Verräumlichung]
  \label{thm:loc-verräumlichung}
  Let $\Psi \in \autautinfn$.
  For every $k \in I_\Psi$ there is a unique map
 $\hat{\Psi}^{(k)} \in \Sym(\Per_k(\sigma_n))$
  such that
  \[\Psi(\varphi)(x) = ((\hat{\Psi}^{(k)})^{-1} \circ \varphi
    \circ\hat{\Psi}^{(k)})(x)\]
  for all $\varphi \in \Simp^{(k)}(\Gamma_n)$ and $x \in
  \Per_k(\sigma_n)$.
  In other words
  \[\rho_k(\Psi(\varphi)) = (\hat{\Psi}^{(k)})^{-1}\circ \rho_k(\varphi) \circ \hat{\Psi}^{(k)}.\]
\end{thm}

Our goal is now to gradually remove the various restrictions in the assumptions of this theorem. We only need the intermediate results for
orientation preserving automorphisms and concentrating on those
simplifies some arguments.

\begin{prop}
\label{prop:sigma-times-two}
Let $\psi \in \orpren$.
Suppose $k$ is a multiple of $\deg(\Psi)$ and $2k \in I_\Psi$. Then the map
  $\rho_{2k}(\sigma_n^k)$ commutes with $\hat{\Psi}^{(2k)}$.
\end{prop}
\begin{proof}
  Set
  $\omega := \hat{\Psi}^{(2k)} \circ \rho_{2k}(\sigma_n^k) \circ
  (\hat{\Psi}^{(2k)})^{-1}$.  Let $\varphi \in
  \Simp^{(k)}(\Gamma_n)$. We define $\pi := \rho_{2k}(\varphi)$, thus
  $\varphi = \nu_{2k}(\pi)$.
  We have
  $\sigma^k_n \circ \varphi = \varphi \circ \sigma^k_n$, hence
 \begin{align*}
     \Psi(\sigma_n^k \circ \varphi) = \Psi(\varphi \circ \sigma_n^k)
   &=\sigma_n^k \circ \Psi(\varphi) = \Psi(\varphi)\circ \sigma_n^k.
 \end{align*}

Applying $\rho_{2k}$ gives
\begin{align*}
    \rho_{2k}(\sigma_n^k) \circ (\hat{\Psi}^{(2k)})^{-1} \circ
  \pi \circ \hat{\Psi}^{(2k)} &=   (\hat{\Psi}^{(2k)})^{-1} \circ
  \pi \circ \hat{\Psi}^{(2k)} \circ  \rho_{2k}(\sigma_n^k) \\
  \hat{\Psi}^{(2k)} \circ \rho_{2k}(\sigma_n^k) \circ (\hat{\Psi}^{(2k)})^{-1} \circ
 \pi &=
  \pi \circ \hat{\Psi}^{(2k)} \circ  \rho_{2k}(\sigma_n^k)
                   \circ (\hat{\Psi}^{(2k)})^{-1} \\
  \omega \circ \pi &= \pi \circ \omega.
\end{align*}
In other words $\omega$ is in the centralizer of
$\rho_{2k}(\Simp^{(k)}(\Gamma_n))$ within $\Sym(\Per_{2k}(\sigma_n))$.
This centralizer consists just of the two elements\footnote{Here it is apparent why the number $2$ is special.} $\id$ and $\rho_{2k}(\sigma^k_n)$.
Since $ \hat{\Psi}^{(2k)} \circ \rho_{2k}(\sigma_n^k) \circ
(\hat{\Psi}^{(2k)})^{-1}=\id_{\Per_{2k}(\sigma_n)}$ would imply
$\rho_{2k}(\sigma^k_n)=\id_{\Per_{2k}(\sigma_n)}$, which is clearly
impossible, we finally have $ \hat{\Psi}^{(2k)} \circ \rho_{2k}(\sigma_n^k) \circ
(\hat{\Psi}^{(2k)})^{-1}=\rho_{2k}(\sigma_n^k)$.
\end{proof}

We are now ready to redeem the promise from the last section.

\begin{prop}
    \label{prop:defective-set-empty}
    Let $\Psi \in \orpren$.
  We have $\defect(\Psi)=\emptyset$, or in other words, $I_\Psi =
  \{k\deg(\Psi) \setsep k \in \bbN, k\deg(\Psi)\geq 3\}$.
\end{prop}
\begin{proof}
  Assume $\defect(\Psi) \neq \emptyset$ and let $k$ be the maximal
  element of $\defect(\Psi)$.  Then $2k \in I_{\Psi}$.  Therefore
  $\hat{\Psi}^{(2k)}(\Per_{k}(\sigma_n))=\Per_k(\sigma_n)$ by the
  previous theorem.  Since $k \in \defect(\Psi)$ there must be a non
  trivial element $\pi \in \Simp^{(k)}(\Gamma_n)$ such that
  $\Psi(\pi)(x)=x$ for all $x \in \Per_k(\sigma_n)$.  But since
  $\Per_k(\sigma_n)\subseteq \Per_{2k}(\sigma_n)$ we have, for $x \in \Per_k(\sigma_n)$,
$$x=\Psi(\pi)(x)=((\hat{\Psi}^{(2k)})^{-1} \circ \pi \circ
  \hat{\Psi}^{(2k)})(x)$$
and hence
$$\hat{\Psi}^{(2k)}(x) = \pi \circ \hat{\Psi}^{(2k)}(x).$$
Since $\hat{\Psi}^{(2k)}(\Per_{k}(\sigma_n))=\Per_k(\sigma_n)$, this implies
$$\pi_{|\Per_{k}(\sigma_n)}=\id_{|\Per_{k}(\sigma_n)}.$$
  Hence $\pi=\id$. This contradicts the non-triviality of
  $\pi$. Therefore our initial assumption was wrong and we have
  shown $\defect(\Psi)=\emptyset$.
\end{proof}

\begin{prop}
    \label{prop:compat-along-powers-of-two}
  Let $\Psi \in \orpren$.
  For every $k \in I_\Psi$ and $m \in \bbN$ the restriction
  of $\hat{\Psi}^{(2^mk)}$ from $\Per_{2^{m}k}(\sigma_n)$ to $\Per_{k}(\sigma_n)$
  equals $\hat{\Psi}^{(k)}$.
\end{prop}
\begin{proof}
    It is enough to show this for $m=1$, the rest follows
    by induction.
    Define the map $\omega: \Per_k(\sigma_n) \to \Per_k(\sigma_n)$
    as $\omega(x) := \hat{\Psi}^{(2k)}((\hat{\Psi}^{(k)})^{-1}(x))$.
    The map $\omega$ is well defined
    since $\Per_k(\sigma_n)$ is invariant
    under $\hat{\Psi}^{(2k)}$ by \Cref{prop:sigma-times-two}.
    Furthermore, $\omega$ must lie in the center
    of $\Sym(\Per_k(\sigma_n))$
    since for $\varphi \in \Simp^{(k)}(\Gamma_{n})$
    and $x \in \Per_k(\sigma_n)$ we have
    \begin{align*}
        \Psi(\varphi)(x) &= (\hat{\Psi}^{(2k)})^{-1}
        \circ \varphi \circ (\hat{\Psi}^{(2k)})) (x) \\
        &=(\hat{\Psi}^{(k)})^{-1}
        \circ \varphi \circ (\hat{\Psi}^{(k)}) (x)
    \end{align*}
    and hence on $\Per_{k}(\sigma_{n})$
    $$\hat{\Psi}^{(2k))} \circ ((\hat{\Psi}^{(k)})^{-1} \circ \varphi \circ (\hat{\Psi}^{(k)}) \circ (\hat{\Psi}^{(2k)})^{-1} = \varphi$$
Since the center of $\Sym(\Per_{k}(\sigma_{n}))$ is trivial, $\omega$ is the identity which proves the proposition.
\end{proof}

\begin{prop}
  \label{prop:verraum-stab-inerts}
  Let $\psi \in \orpren$.
  Let $k \in I_\Psi$ and suppose $\gamma \in \Inert^{(k)}(\sigma_n)$.  Then
  $\rho_{2k}(\Psi(\gamma))=(\hat{\Psi}^{(2k)})^{-1} \circ
  \rho_{2k}(\gamma)\circ \hat{\Psi}^{(2k)}$.
\end{prop}
\begin{proof}
  By \Cref{lem:boyle-rep} there are $\pi_1, \pi_2 \in
  \Simp^{(2k)}(\Gamma_n)$
  such that $\gamma = \sigma_n^{-k} \circ \pi_1 \circ \sigma_n^k \circ
  \pi_2$.
  Applying $\Psi$ to this equation gives us:
  \begin{align*}
    &\rho_{2k}(\Psi(\gamma)) = \rho_{2k}(\sigma_n^{-k}) \circ
                              \rho_{2k}(\Psi(\pi_1)) \circ \rho_{2k}(\sigma_n^{k})\circ
                              \rho_{2k}(\Psi(\pi_2)) \\
                            &=\rho_{2k}(\sigma_n^{-k})
                              \circ (\hat{\Psi}^{(2k)})^{-1}
                              \circ \rho_{2k}(\pi_1)
                              \circ \hat{\Psi}^{(2k)}
                              \circ \rho_{2k}(\sigma_n^{k})
                              \circ (\hat{\Psi}^{(2k)})^{-1}
                              \circ \rho_{2k}(\pi_2)
                              \circ \hat{\Psi}^{(2k)}\\
                            &=(\hat{\Psi}^{(2k)})^{-1}
                              \circ \rho_{2k}(\sigma_n^{-k})
                              \circ \rho_{2k}(\pi_1)
                              \circ \rho_{2k}(\sigma_n^{k})
                              \circ \rho_{2k}(\pi_2)
                              \circ \hat{\Psi}^{(2k)}\\
    &=(\hat{\Psi}^{(2k)})^{-1}
                              \circ \rho_{2k}(\gamma)
                              \circ \hat{\Psi}^{(2k)}\qedhere
  \end{align*}
\end{proof}

Putting what we have together now gives the following.
\begin{prop}
    Let $\Psi \in \orpren$.
  Let $k \in I_\Psi$ and $m \in \bbN$. Then for
  every $x \in \Per_{2^mk}(\sigma_n)$ and $\gamma \in \Inert^{(k)}(\sigma_n)$
  we have
  \begin{align*}
    \Psi(\gamma)(x)&=((\hat{\Psi}^{(2^mk)})^{-1} \circ \gamma\circ
    \hat{\Psi}^{(2^mk)})(x), \\
    \sigma_n^{k}(x) &= ((\hat{\Psi}^{(2^mk)})^{-1} \circ \sigma_n^{k}\circ
    \hat{\Psi}^{(2^mk)})(x).
  \end{align*}
\end{prop}

  \newcommand{\spa}{\spatial^{(k)}}
  \newcommand{\spainv}{(\spa)^{-1}}
By using yet another centralizer calculation we can upgrade this to
the following.
\begin{prop}
  \label{prop:commuting-with-the-shift-arb-ind}
  Let $\Psi \in \orpren$.
  Let $k \in I_{\Psi}$. Then for all $x \in \Per_{k}(\sigma_n)$
  \begin{align*}
    \sigma^{\deg(\Psi)}_n(x) = ((\spatial^{(k)})^{-1} \circ
    \sigma_n^{\deg(\Psi)} \circ \spatial^{(k)})(x).
  \end{align*}
\end{prop}
\begin{proof}
  Let $\varphi \in \Simp^{(k)}(\Gamma_n)$ and set $\alpha := \sigma_n^{\deg(\Psi)}$.
  Note that $\alpha^{-1} \circ
  \varphi \circ \alpha \in \Inert^{(k)}(\sigma_n)$ and
  $$\Psi(\alpha^{-1} \circ \varphi \circ \alpha) = \Psi(\alpha^{-1})
  \circ \Psi(\varphi) \circ \Psi(\alpha) = \alpha^{-1} \circ \Psi(\varphi)
  \circ \alpha.$$
  Using this together with the previous theorem we can derive the following for every $x \in \Per_{k}(\sigma_n)$:
  \begin{align*}
    &(\spainv \circ \alpha^{-1} \circ \varphi \circ \alpha \circ
    \spa)(x)
    = (\alpha^{-1} \circ \spainv \circ \varphi \circ \spa \circ
      \alpha)(x).
  \end{align*}
  For $y=(\alpha \circ\spa)(x)$ this means
  \begin{align*}
  (\spainv \circ \alpha^{-1} \circ \varphi)(y)
    &= (\alpha^{-1} \circ \spainv \circ \varphi \circ \spa \circ \alpha
      \circ \spainv \circ \alpha^{-1})(y)
    \end{align*}
    and hence
    \begin{align*}
    (\spa \circ \alpha \circ \spainv \circ \alpha^{-1} \circ \varphi)(y)
    &= (\varphi \circ \spa \circ \alpha
      \circ \spainv \circ \alpha^{-1})(y).
  \end{align*}
  Since both $\alpha$ and $\hat{\Psi}^{(k)}$ map $\Per_{k}(\sigma_{n})$ to itself, so does $\alpha \circ \hat{\Psi}^{(k)}$. Therefore $\spa \circ \alpha \circ \spainv \circ \alpha^{-1}$
  is in the centralizer of the image of $\Simp^{(k)}(\Gamma_n)$ in
  $\Sym(\Per_{k}(\sigma_n))$. Since the image of $\Simp^{(k)}(\Gamma_n)$ is all of $\Sym(\Per_{k}(\sigma_n))$ and the center of $\Sym(\Per_{k}(\sigma_{n}))$ is trivial, it follows that $\spa \circ \alpha \circ \spainv \circ \alpha^{-1}$
  is the identity on $\Per_k(\sigma_n)$.
  In other words, for all $x \in \Per_k(\sigma_n)$
  we have
  \begin{align*}
        \sigma^{\deg(\Psi)}_n(x) &= ((\spatial^{(k)})^{-1} \circ
    \sigma_n^{\deg(\Psi)} \circ \spatial^{(k)})(x). \qedhere
  \end{align*}
\end{proof}

\begin{prop}
    \label{prop:compat-of-all}
    Let $\Psi \in \orpren$.
  For every $k \in I_\Psi$ and $\ell \in \bbN$ the restriction of
  $\spatial^{(\ell k)}$ from $\Per_{\ell k}(\sigma_n)$ to
  $\Per_{k}(\sigma_n)$ equals $\spatial^{(k)}$.
\end{prop}
\begin{proof}
    The argument is the same as in the proof of \Cref{prop:compat-along-powers-of-two}.
\end{proof}

This proposition now finally allows us to stitch together all the maps
$\spatial^{(k)}$.

\begin{prop}
    \label{thm:global-spatialization-inert-and-shift}
  For every $\Psi \in \autautinfn$
  there exists a unique bijective map $\spatial: \Per(\sigma_n) \to \Per(\sigma_n)$
  such that for every $x \in \Per(\sigma_n)$ and
  $\varphi \in \Inert^{(\infty)}(\sigma_n)$ we have
  \begin{align*}
    \Psi(\varphi)(x)&=(\spatial^{-1} \circ \varphi \circ \spatial)(x),
    \\
    \sigma_n^{\deg(\Psi)} &=  (\spatial^{-1} \circ \sigma_n^{\deg(\Psi)} \circ \spatial)(x).
  \end{align*}
\end{prop}
\begin{proof}
    First assume that $\Psi$ is orientation preserving.
    For $x \in \Per(\sigma)$
    there is $k \in I_\Psi$ for which
    $x \in \Per_k(\sigma)$.
    Define $\hat{\Psi}(x):=\hat{\Psi}^{(k)}(x)$.
    By \Cref{prop:compat-of-all} this map is well-defined
    and has the desired properties by
    \Cref{prop:commuting-with-the-shift-arb-ind}
    and
    \Cref{prop:verraum-stab-inerts}.
    Uniqueness follows from the uniqueness
    of the local \verraum maps $\Psi^{(k)}$.

    Now consider an automorphism $\Psi$ which is not
    orientation preserving. Then $\Psi':=\Xi_{n} \circ \Psi$ is orientation preserving and for
    $\varphi \in \inertinfn \cup \{\sigma_n^{\deg(\Psi)}\}$ and $x \in \Per(\sigma_n)$ we get
    \begin{align*}
        \Psi(\varphi)(x) &= (\Xi_{n}^{-1} \circ \Psi')(\varphi)(x) \\
        &=\xi \circ \hat{\Psi'}^{-1} \circ \varphi \circ \hat{\Psi'} \circ \xi^{-1} (x)
    \end{align*}
    Hence we can set
    $\hat{\Psi}:=\hat{\Psi'} \circ \xi^{-1}$.
\end{proof}

We call $\hat{\Psi}$ the \emph{Verräumlichung} of $\Psi$.

\begin{lem}
    \label{lem:Psi-of-ker-rho}
    Suppose $\Psi \in \aut(\autinfn)$ and let $\ell \in I_\Psi$.
   Then the subgroup $\ker \rho_\ell \cap \inertinfn$
    of $\autinfn$ is invariant under $\Psi$.
\end{lem}
\begin{proof}
    First we show $\Psi(\ker \rho_\ell \cap \inertinf) \subseteq \ker \rho_\ell \cap \inertinfn$.
    Since $\inertinfn$ is the commutator subgroup of $\autinfn$
    it is invariant under $\Psi$.
    Suppose that $\varphi \in \ker \rho_\ell \cap \inertinfn$
    and let $x \in \Per_\ell(\sigma_n)$.
    Then $\hat{\Psi}(x) \in \Per_\ell(\sigma_n)$ and $\varphi(\hat{\Psi}(x))=\hat{\Psi}(x)$
    and hence
    \begin{align*}
        \Psi(\varphi)(x)&= \hat{\Psi}^{-1} \circ \varphi \circ \hat{\Psi}(x) \\
        &=x.
    \end{align*}
    Now we apply the inclusion shown so far to $\Psi^{-1}$
    and get the reverse inclusion.
\end{proof}

\begin{thm}[Global \verraum]
\label{thm:prop-global-verraum}
For every $\Psi \in \autautinfn$
  the associated \verraum map $\spatial: \Per(\sigma_n) \to \Per(\sigma_n)$
  satisfies
  \begin{align*}
    \Psi(\varphi)(x)&=(\spatial^{-1} \circ \varphi \circ \spatial)(x)
  \end{align*}
  for all $\varphi \in \autinfn$ and $x \in \Per(\sigma_n)$.
\end{thm}
\begin{proof}
    First assume that $\Psi$ is orientation preserving.
    Let $\ell \in I_\Psi$ and let $x \in \Per_\ell(\sigma_n)$.
    There is $\tau \in \Simp^{(\ell)}(\Gamma_n)$
    such that $\rho_\ell(\tau)=\rho_\ell(\varphi)$.
    Then $\rho_\ell(\varphi\circ \tau^{-1})$ is the identity
    and therefore commutes with every element $\eta \in \rho_\ell(\autinfn)$.
    In other words the commutator of $\varphi \circ \tau^{-1}$
    and $\eta$ is in $\inertinfn \cap \ker \rho_\ell$.
    By \Cref{lem:Psi-of-ker-rho}
    this implies that the commutator
    of $\Psi(\varphi)\circ \Psi(\tau^{-1})$ and $\Psi(\eta)$
    is in the kernel of $\rho_\ell$.
    Hence $\rho_\ell(\Psi(\varphi)\circ \Psi(\tau^{-1}))$ is in the
    center of $\rho_\ell(\Psi(\autinfn)) = \Sym(\Per_\ell(\sigma_n))$.
    The center of the later group is trivial, hence
    $\rho_\ell(\Psi(\varphi))=\rho_\ell(\Psi(\tau))$.
    In other words, for $x \in \Per_\ell(\sigma_n)$
    we get
    \begin{align*}
        \Psi(\varphi)(x) &= \Psi(\tau)(x) = \hat{\Psi}^{-1}\circ \tau \circ \hat{\Psi}(x) \\
        &= \hat{\Psi}^{-1} \circ \varphi \circ \hat{\Psi}(x). \qedhere
    \end{align*}
    The case that $\Psi$ is not orientation preserving works exactly as
    in the proof of \Cref{thm:global-spatialization-inert-and-shift}.
\end{proof}

\subsection{The \verraum is a homomorphism}
We can now prove that the \verraum map is an injective homomorphism. 


\begin{thm}\label{prop:verraum-injective}
  The map  $\autautinfn \to \Sym(\Per(\sigma_n)), \Psi \mapsto \hat{\Psi}$
  is an injective homomorphism.
\end{thm}
\begin{proof}
  Consider $\Psi, \Upsilon \in \orpren$
  and set $\Lambda = \Psi \circ \Upsilon$.
  We want to show that
  $\hat{\Lambda} = \hat{\Psi}\circ \hat{\Upsilon}$.
  Keep in mind that we don't know if $ \hat{\Psi}\circ \hat{\Upsilon}$
  is a \verraum. Nevertheless we know that on periodic points
  $\hat{\Lambda}^{-1} \circ \varphi \circ \hat{\Lambda} =
  \Lambda(\varphi)=
  \hat{\Upsilon}^{-1} \circ \hat{\Psi}^{-1} \circ \varphi \circ \hat{\Psi}
  \circ \hat{\Upsilon}$
  for all $\varphi \in \autinfn$.
  So $ \hat{\Psi}
  \circ \hat{\Upsilon} \circ \hat{\Lambda}^{-1}$ is in the centralizer
  of $\autinfn$ within $\Sym(\Per(\sigma_n))$
  which is trivial by
  \Cref{cor:centralizer-in-per}.
  Injectivity follows by \Cref{thm:prop-global-verraum}.
\end{proof}

\subsection{The profinite part of $\degonen$}
Armed with the \verraum on periodic points, we use it now to deduce some more properties of the embedding $\mathcal{N} \colon \hat{\mathbb{Z}} \to \aut(\autinfn)$.
\begin{prop}
  \label{prop:shifting-verraum}
  Let $\Psi \in \degonen$. Then $\Psi$ is in the image $\mathcal{N}$ if and only if there is a function $k: \Per(\sigma_n)\to \bbZ$ such that $\hat{\Psi}(x)=\sigma^{k(x)}(x)$ for all $x \in \Per(\sigma_n)$.
\end{prop}
\begin{proof}
Suppose first that $\Psi$ is in the image of $\mathcal{N}$, say $\Psi = \profact(a)$ with $a=(a_k)_{k \in \bbN}$. Then for any $\pi \in \Simp^{(k)}(\Gamma_{n})$ we have $\Psi(\pi) = \mathcal{N}(a)(\pi) = \sigma^{-a_{k}}\circ \pi \circ \sigma^{a_{k}}$. Thus $\hat{\Psi}(x) = \sigma^{a_{k}}(x)$ for all $x \in \Per_{k}(\sigma_{n})$, by definition of the \verraum.

For the other direction, we will show that there is a sequence $(m_\ell)_{\ell \in \bbN}$ such that
    $\hat{\Psi}(x) = \sigma^{m_\ell}(x)$
    for every $x \in \Per_\ell(\sigma_n)$. To see that this is enough consider
    $k \mid \ell$. Then $\Per_k(\sigma_n) \subseteq \Per_\ell(\sigma_n)$.
    Therefore $\sigma^{m_k}(x) = \sigma^{m_\ell}(x)$ for all
    $x \in \Per_k(\sigma_n)$. Hence $m_k \equiv m_\ell \text{ mod } k$.
    This shows that $(m_\ell)_{\ell \in \bbN}$ is in  $\profint$
    and therefore $\Psi$ is in the image of $\profact$.

    Notice that we may assume that for every periodic point $x$ we have $k(x) \in \{0,\dots,\per(x)-1\}$
    where $\per(x)$ is the minimal period of $x$.
    Next notice that for $x \in \Per(\sigma_n)$
    and $\ell \in \bbZ$
    we know that $k(\sigma^\ell(x))=k(x)$
    since
    \begin{align*}
        \sigma^{k(\sigma^\ell(x))}(\sigma^{\ell}(x))
        &=\hat{\Psi}(\sigma^\ell(x))\\
        &=\sigma^\ell(\hat{\Psi}(x))\\
        &=\sigma^{\ell+k(x)}(x) = \sigma^{k(x)}(\sigma^{\ell}(x))
    \end{align*}
    Now let  $x,y \in \Per_\ell(\sigma_n)$ be points
    in different $\sigma_n$-orbits, both having minimal
    period $\ell \geq 2$.
    By \Cref{lem:pretty-trans}
    there is an automorphism $\varphi \in \Aut(\sigma)$
    which induces a permutation
    on $\Per_\ell(\sigma_n)$ containing a cycle
    $(z_1,\dots,z_{\ell+1})$
    with $z_1=x$ and $z_2 = y$.
    There are enough points of minimal
    period $\ell$ in $\Per_\ell(\sigma_n)$
    since the case $\ell=2$ and $n=2$ (the only problematic case)
    can not happen, as there is only one  $\sigma_n$
    orbit of cardinality $2$ in the full $2$-shift.
   Now for all $u \in \Per(\sigma)$ we have $\Psi(\varphi)(u) = \hat{\Psi}^{-1} \circ \varphi \circ \hat{\Psi}(u)$, so
    \begin{align*}
        \Psi(\varphi)(u) &=\sigma^{-k(\sigma^{k(u)}(\varphi( u)))}(\varphi( \sigma^{k(u)}(u))) \\
        &=\sigma^{-k(\varphi(u))}(\varphi( \sigma^{k(u)}(u))) \\
        &=\sigma_n^{-k(\varphi(u))+k(u)}(\varphi(u))
    \end{align*}
    Setting $v := \varphi(u)$
    we obtain
    \begin{align*}
        \Psi(\varphi)(\varphi^{-1}(v))=\sigma^{-k(v)+k(\varphi^{-1}(v))}(v)
    \end{align*}
    Since $\Psi(\varphi) \circ \varphi^{-1}$
    is in $\Aut(\sigma)$
    there must be $r_\varphi$ such that
    $\sigma^{-k(v)+k(\varphi^{-1}(v))}(v)=\sigma^{r_\varphi}(v)$
    for all $v \in \Per(\sigma)$ by \Cref{thm:shifting-has-to-be-uniform}.
    In particular
    $k(x) \equiv r_\varphi+k(y) \text{ mod } \ell$.
    Arguing along the cycle $(z_1,\dots,z_{\ell+1})$
    we get
    $k(x) \equiv k(x) + (\ell+1)r_\varphi \text{ mod } \ell$,
    and hence $r_\varphi \equiv 0 \text{ mod } \ell$.
    Therefore $k(x) \equiv k(y) \text{ mod } \ell$.

    Thus for every $q \in \bbN$ there is
    $k_q$ such that
    $\Psi(\varphi)(x)=\sigma^{-k_q}(\varphi(\sigma^{k_q}(x)))$
    for all $x$ with minimal $\sigma$-period $q$.

    For every $\ell$ pick $m_\ell \in \{0,\dots,\ell-1\}$
    such that for infinitely many $q$ we have $k_q \equiv m_\ell \text{ mod }\ell$.
    We claim that
    $\hat{\Psi}(v)=\sigma^{m_\ell}(v)$
    for all $v \in \Per_\ell(\sigma)$.
    We can see this as follows:
    Fix $\ell$ and let $\varphi \in \Aut(\sigma^\ell)$
    be the simple automorphism
    which cyclically permutes the symbol at the origin
    and fixes the $\ell-1$ symbols
    at coordinates $1$ to $\ell-1$.
    Consider the map $\beta:=\sigma^{m_\ell} \circ \Psi(\varphi) \circ\sigma^{-m_\ell}\circ \varphi^{-1}$.
    There is a dense set
    of $\sigma$-periodic points which
    are fixed by $\beta$, hence
    $\beta$ is the identity.
    Therefore for every $k \in \bbN$ we have $\varphi(x) = (\sigma^{m_\ell-k_q} \circ \varphi \circ \sigma^{m_\ell-k_q})(x)$ for $x \in \Per_\ell(\sigma)$
    with minimal period $q$.
    The definition of $\varphi$ implies that $m_\ell - k_q \equiv \ell \equiv 0 \text{ mod } q$.
    This shows $\hat{\Psi}(x) = \sigma^{m_\ell}(x)$
    for all $x \in \Per_\ell(\sigma)$.
\end{proof}

\begin{rem}
    The previous proof also shows that for $a=(a_\ell)_{\ell \in \bbN} \in \profint$ we have $\widehat{\mathcal{N}(a)}(x) = \sigma^{a_\ell}$ for $x \in \Per_\ell(\sigma_n)$.
\end{rem}

Recall by Theorem~\ref{thm:centralizer-aut-in-per} that, for $n \ge 3$, the centralizer of the image of $\autn$ in $\Sym(\Per(\sigma_{n}))$ is contained the set of bijections which preserve $\sigma_{n}$-orbits for $n \ge 3$, and for $n=2$ there is one exceptional symmetry which does not preserve $\sigma_{2}$-orbits. We need one short lemma to handle this exceptional case before moving to the next theorem.
\begin{lem}
  \label{lem:centralizer-2-verraeumlichung-special-case}
  No map $\psi \in \Sym(\Per(\sigma_2))$
  which preserves all $\sigma_2$-orbits of length greater or equal $2$
  and exchanges the two fixed points is the \verraum of an
  element of $\degonen$.
\end{lem}
\begin{proof}
  Consider $\varphi \in \Simp^{(2)}(\Gamma_2)$
  given by the cycle $11\rightarrow 00 \rightarrow 01 \rightarrow 10$.
  Assume there was $\Psi \in \degonen(\autinf(\sigma_2))$
  such that $\hat{\Psi}$ has the properties in the statement
  of the lemma.
  Then $\Psi(\varphi)(\linf 0 \rinf)=\linf 1 \rinf$.
  Set $z:=\linf(0^{2k}010^{2k})\rinf$.
  For sufficiently large $k>0$ the point
  $\Psi(\varphi)(z)$ contains
  the pattern $1111$.
  But $\Psi(\varphi)(z)=\hat{\Psi}^{-1}\circ \varphi\circ \hat{\Psi}(z)$ is in the $\sigma_2$-orbit of $y:=\varphi(\sigma^\ell(z))$
  for some $\ell \in \bbZ$.
  Since $\varphi$ commutes with $\sigma_2^2$,
  the point $y$ is in the $\sigma_2$-orbit of either
  \begin{align*}
  \varphi(\linf(0^{2k}010^{2k})\rinf) &= \linf((01)^{k}10(01)^{k})\rinf,\text{ or}\\
    \varphi(\sigma(\linf(0^{2k}010^{2k})\rinf))&=
  \varphi(\linf(0^{2k}100^{2k})\rinf) =
    \linf((01)^{k}11(01)^{k})\rinf
  \end{align*}
 neither of which contains the pattern $1111$.
\end{proof}

Observe that if $\Psi \in \degonen$ then $\Psi(\aut(\sigma_{n})) = \aut(\sigma_{n})$. Hence there is a homomorphism $\degonen \to \autautn$ obtained by restricting any $\Psi \in \degonen$ to $\aut(\sigma_{n})$.

\begin{thm}
The subgroup $\profact(\profint)$ is the kernel of the homomorphism $\degonen \to \autautn$.
\end{thm}
\begin{proof}
It is clear that $\profact(\profint)$ is in the kernel of this homomorphism. Now suppose $\Psi \in \degonen$ satisfies $\Psi(\varphi)=\varphi$ for all $\varphi \in \autn$. Then $\hat{\Psi}$ is in the centralizer of the image of $\autn$ in $\Sym(\Per(\sigma_n))$.
  By \Cref{thm:centralizer-aut-in-per}
  and \Cref{lem:centralizer-2-verraeumlichung-special-case}
  this implies $\hat{\Psi}$
  preserves $\sigma_n$ orbits. By \Cref{prop:shifting-verraum}
  this implies $\Psi \in \profact(\profint)$, and the result follows.
\end{proof}


Later we will realize the image of $\mathcal{N}$ as a kernel of another homomorphism.

\begin{rem}
The subgroup $\profact(\profint)$ is not a normal subgroup of
  $\autautinfn$,
  as already $\varphi^{-1} \circ \sigma_n \circ \varphi$
  is in general not in $\profact(\profint)$
  for $\varphi \in \Simp^{(2)}(\sigma_n)$. In particular, there is no analogous version of Theorem~\ref{thm:exactsequenceaut1} for all of $\aut(\autinfn)$.
\end{rem}

\section{Continuity properties of the \verraum}
In the previous sections, we constructed the \verraum $\hat{\Psi}$ of an isomorphism $\Psi \colon \autinfn \to \autinfn$. Our goal now is to obtain the continuity of $\hat{\Psi}$ at the level of the space of chain recurrent subshifts.

Recall that $\mathcal{S}(\sigma_{n})$ denotes the space of all subshifts of $X_{n}$ with the Hausdorff metric. For a subshift $(X,\sigma_{X})$ and $k \in \mathbb{N}$ we let $\mathcal{L}_{k}(X)$ denote the set of words of length $k$ that appear in $X$.
The Hausdorff metric on $\mathcal{S}(X_{n})$ is equivalent to the metric defined by
$$d(X,Y) = 2^{-\inf\{k \mid \mathcal{L}_{k}(X) \ne \mathcal{L}_{k}(Y)\}}.$$
This and other fundamental properties of $\mathcal{S}(X_{n})$ can be found in~\cite{PSgeneric}.

We will also make use of the space of all subgroups of $\autinfn$, which we define by
$$\sub(\autinf(\sigma_{n})) = \{H \mid H \textrm{ is a subgroup of }\autinfn\}.$$
This is a compact space with the Chabauty topology; since $\autinfn$ is countable, it is metrizable. It is straightforward to check that if $\Psi \colon \autinfn \to \autinfn$ is a group isomorphism, then the map $\sub(\autinfn) \to \sub(\autinfn)$ given by $H \mapsto \Psi(H)$ is a homeomorphism.

For a subshift $Y \subseteq X_{n}$ we define the pointwise stabilizer subgroup
$$\stp(Y) = \{\varphi \in \autinfn \mid \varphi(x) = x \textrm{ for all } x \in Y\}.$$
We can then define the map
\begin{align*}
\stp \colon& \mathcal{S}(\sigma_{n}) \to \sub(\autinfn)\\
&Y \mapsto \stp(Y).
\end{align*}
A key tool for us to obtain some continuity is the following theorem.
\begin{thm}\label{thm:contstpmap}
Suppose $Y_{m} \to Y$ in $\mathcal{S}(\sigma_{n})$. Then $\stp(Y_{m}) \to \stp(Y)$ in $\sub(\autinf(\sigma_{n}))$. In other words, the map $\stp \colon \mathcal{S}(\sigma_{n}) \to \sub(\autinf(\sigma_{n}))$ is continuous.
\end{thm}
\begin{proof}
First suppose $Y_{m} \to Y$ in $\mathcal{S}(X_{n})$ and let $\alpha \in \stp(Y)$. Suppose $\alpha \in \aut(\sigma_{n}^{k})$. By the Curtis-Hedlund-Lyndon Theorem~\cite{Hedlund1969} there exists a natural number $R$ such that if $y_{[-Rk,(R+1)k-1]} = y^{\prime}_{[-Rk,(R+1)k-1]}$ then $\alpha(y)_{[0,k-1]} = \alpha(y^{\prime})_{[0,k-1]}$. Since $Y_{m} \to Y$, there exists $M$ such that for all $m \ge M$ we have $\mathcal{L}_{k(2R+1)}(Y_{m}) = \mathcal{L}_{k(2R+1)}(Y)$. Suppose $w$ is any $k(2R+1)$-word in $Y_{m}$. Then there exists $y \in Y$ such that $y_{[-Rk,(R+1)k-1]} = w$. Since $\alpha \in \stp(Y)$, the $k$-word $\alpha(y)_{[0,k-1]}$ coincides with $y_{[0,k-1]}$. This implies $\alpha(y) = y$ for all $y \in Y_{m}$, so $\alpha \in \stp(Y_{m})$ for all $m \ge M$.

On the other hand, suppose $\beta \not \in \stp(Y)$. Then there exists $y \in Y$ such that $\beta(y) \ne y$, and a similar argument to the above implies there exists $M$ such that for all $m \ge M$ there exists $z_{m} \in Y_{m}$ such that $\beta(z_{m}) \ne z_{m}$, and hence $\beta \not \in \stp(Y_{m})$ for all $m \ge M$.
\end{proof}

For a subgroup $H \subseteq \autinfn$, define
$$\textrm{Fix}(H) = \{x \in X_{n} \mid \alpha(x) = x \textrm{ for all } \alpha \in H\}.$$

\begin{lem}
\label{lem:fix-stab-eq-id}
For any subshift $Y \subseteq X_{n}$, $\textrm{Fix}(\stp(Y)) = Y$.
\end{lem}
\begin{proof}
If $y \in Y$ then $\alpha(y) = y$ for all $\alpha \in \stp(Y)$, so $y \in \textrm{Fix}(\stp(Y))$. Thus $Y \subseteq \textrm{Fix}(\stp(Y))$. Now suppose $z \not \in Y$. Then there exists a word $v$ of length $\ell$ that appears in $z$ such that $z \not \in \mathcal{L}_\ell(Y)$. Let $a,b$ be different symbols such that $va$ is a word in $z$ of length $\ell+1$ and define the simple automorphism $\tau \in \simp^{(\ell+1)}(\Gamma_{n})$ which permutes $va$ and $vb$ and does nothing else. Since $va$ appears in $z$, $\tau(z) \ne z$. Moreover, since $v$ does not appear in $Y$, $\tau$ acts by the identity on $Y$, and hence $\tau \in \stp(Y)$. Thus $z \not \in \textrm{Fix}(\stp(Y))$, so $\textrm{Fix}(\stp(Y)) \subseteq Y$, and altogether we have that $\textrm{Fix}(\stp(Y)) = Y$.
\end{proof}

\begin{thm}\label{thm:subsforlimits}
Suppose $\stp(Y_{m}) \to H$ in $\sub(\autinf(\sigma_{n}))$. Then $Y_{m}$
converges to some subshift $Y$ in $\mathcal{S}(X_{n})$ and we have $H = \stp(Y)$.
\end{thm}

\begin{proof}
First suppose $Y$ is a limit point of $Y_{m}$, so there exists a subsequence $m_{k}$ such that $Y_{m_{k}} \to Y \in \mathcal{S}(\sigma_{n})$. At least one such $Y$ exists since $\mathcal{S}(\sigma_{n})$ is compact. By assumption we have $\stp(Y_{m_k}) \to H$
and by \Cref{thm:contstpmap} we have $\stp(Y_{m_k}) \to \stp(Y)$.
Hence $H = \stp(Y)$.

Now suppose $Z_{1}$ and $Z_{2}$ are both limits point of $Y_{m}$. By the above we have $\stp(Z_{1}) = H = \stp(Z_{2})$, and hence by \Cref{lem:fix-stab-eq-id} we have
$$Z_{1} = \textrm{Fix}(\stp(Z_{1})) = \textrm{Fix}(\stp(Z_{2})) = Z_{2}.$$
Thus $Y_{m}$ has precisely one limit point $Y$, so $Y_{m} \to Y$.
\end{proof}

Recall $\mathcal{CR}(\sigma_{n})$ denotes the subset of $\mathcal{S}(X_{n})$ consisting of chain recurrent subshifts. The following proves what was stated in the introduction: that $\mathcal{CR}(\sigma_{n})$ is precisely the closure of the finite systems in $\mathcal{S}(X_{n})$.
\begin{prop}\label{prop:CRlimitperiodic}
Let $\mathcal{CR}(\sigma_n) \subseteq \mathcal{S}(\sigma_n)$
    be the space of chain recurrent subshifts of the full $n$-shift.
    Then $\mathcal{CR}(\sigma_n)$ is the closure of the finite
    subshifts.
\end{prop}
\begin{proof}
Suppose that $Y$ is a chain recurrent subshift in $X_{n}$. Then by \cite[Proposition 3.2]{boyleSoficShiftsCannot2004}, the $n$th Markov approximation $Y^{(n)}$ of $Y$ is a nonwandering shift of finite type. Since the periodic points are dense in a nonwandering shift of finite type, the subshift $Y^{(n)}$ is a limit of finite subshifts. But $Y^{(n)}$ converges to $Y$ in the space of subshifts, so it follows that $Y$ is in the closure of the finite systems, and hence the set of chain recurrent subshifts is contained in the closure of the set of finite subshifts.

On the other hand, clearly finite systems are chain recurrent, so it suffices to prove that the set of chain recurrent subshifts is closed. Suppose then that $Z$ is a subshift which is a limit of a sequence of chain recurrent subshifts $Y_{n}$. Then for all $m \ge 1$, there exists $n(m)$ such that the $m$-languages for $Y_{n(m)}$ and $Z$ coincide. It follows that the $m$th Markov approximations $Z^{(m)}$ of $Z$ and $Y^{(m)}$ of $Y_{n(m)}$ coincide. Since $Y_{n(m)}$ is chain recurrent, $Y^{(m)}$, and hence $Z^{(m)}$, is nonwandering. Since this holds for all $m$, this implies $Z$ is chain recurrent, and hence the set of chain recurrent subshifts is closed.
\end{proof}
We record a small lemma to be used in the next theorem.
\begin{lem}\label{lem:verraum-on-finite-sub}
    Let $Q_n$ be a finite subsystem of $(X_n,\sigma_n)$
    and suppose that $\Psi \in \degonen$.
    Then $\Psi(\stp(Q_n))=\stp(\hat{\Psi}(Q_n))$.
\end{lem}
\begin{proof}
    That $\varphi \in \stp(\hat{\Psi}(Q_n)$
    is equivalent to $\varphi(\hat{\Psi}(x))=\hat{\Psi}(x)$
    for all $x \in Q_n$.
    This is equivalent to $\Psi(\varphi)(x)=\hat{\Psi}^{-1} (\varphi(\hat{\Psi}(x)))=x$ for $x \in Q_n$ which in turn is the same as $\varphi \in \Psi(\stp(Q_n))$.
\end{proof}

We now use the fact that $\Psi$ induces a self-homeomorphism of $\sub(\autinfn)$ to extend $\Psi$ to pointwise stabilizer subgroups of chain recurrent subshifts.

\begin{thm}\label{thm:lptolp}
Let $\Psi \in \aut_{1}(\autinfn)$ and suppose $Y \in \mathcal{CR}(\sigma_{n})$. Then there exists a unique subshift $Z \in \mathcal{CR}(\sigma_{n})$ such that $\Psi(\stp(Y)) = \stp(Z)$.
\end{thm}
\begin{proof}
Since $Y \in \mathcal{CR}(\sigma_{n})$, by Proposition~\ref{prop:CRlimitperiodic}, we may choose a sequence $Q_{n}$ of finite systems converging to $Y$ in $\mathcal{S}(\sigma_{n})$. Then $\stp(Q_{n}) \to \stp(Y)$, and by \Cref{lem:verraum-on-finite-sub}, we have that $\Psi(\stp(Q_{n})) = \stp(Q_{n}^{\prime})$ for some finite systems $Q_{n}^{\prime}$ in $X_{n}$. Thus $\stp(Q_{n}^{\prime}) \to \Psi(\stp(Y))$, so by \Cref{thm:subsforlimits}, there exists $Z \in \mathcal{S}(\sigma_{n})$ such that
$Q_n' \to Z$ and $\stp(Z)=\Psi(\stp(Y))$. As a limit of finite systems we have that $Z \in \mathcal{CR}(\sigma_n)$.
Uniqueness of $Z$ follows from $Z=\Fix(\stp(Y))=\Fix(\Psi(\stp(Y)))$ and
\Cref{lem:fix-stab-eq-id}.

\end{proof}

We can now extend the Verräumlichung to $\mathcal{CR}(\sigma_{n})$.

\begin{prop}
Let $\Psi \in \degonen$. The Verräumlichung map extends to a homeomorphism
$$\hat{\Psi} \colon \mathcal{CR}(\sigma_{n}) \to \mathcal{CR}(\sigma_{n})$$
as follows. Given $Y \in \mathcal{CR}(\sigma_{n})$, by Theorem~\ref{thm:lptolp} there exists a unique subshift $Z \in \mathcal{CR}(\sigma_{n})$ such that $\Psi(\stp(Y)) = \stp(Z)$, and we define
$$\hat{\Psi}(Y) = Z.$$
\end{prop}
\begin{proof}
It suffices to prove that $\hat{\Psi}$ is continuous. Suppose $Y_{n} \to Y$ in $\mathcal{CR}(\sigma_{n})$, let $Z_{n} = \hat{\Psi}(Y_{n})$ and $Z = \hat{\Psi}(Y)$. Then by Theorem~\ref{thm:contstpmap} $\stp(Y_{n}) \to \stp(Y)$. Since $\Psi \colon \sub(\autinfn) \to \sub(\autinfn)$ is continuous, $\Psi(\stp(Y_{n})) \to \Psi(\stp(Y))$. By definition of $\hat{\Psi}$, we have $\Psi(\stp(Y_{n})) = \stp(Z_{n})$ and $\Psi(\stp(Y)) = \stp(Z)$. It follows that $\stp(Z_{n}) \to \stp(Z)$. Theorem~\ref{thm:subsforlimits} then gives $Z_{n} \to Z$.
\end{proof}

The group $\aut(\sigma_{n})$ induces an action on $\mathcal{S}(X_{n})$ given by
$$\alpha \colon Y \mapsto \alpha(Y),$$
and this action preserves the subspace $\mathcal{CR}(\sigma_{n})$ of chain recurrent subsystems. Given $\alpha \in \aut(\sigma_{n})$, recall we let $\alpha_{\scriptscriptstyle \mathcal{CR}}$ denote the respective map defined by $\alpha$ on $\mathcal{CR}(\sigma_{n})$.

\begin{thm}
Let $\Psi \colon \autinfn \to \autinfn$ be a degree one isomorphism and let $\hat{\Psi} \colon \mathcal{CR}(\sigma_{n}) \to \mathcal{CR}(\sigma_{n})$ be the associated Verräumlichung homeomorphism. Then $\Psi$ is inner for the action of $\aut(\sigma_{n})$ on $\mathcal{CR}(\sigma_{n})$, in the sense that for any $\alpha \in \aut(\sigma_{n})$
$$\Psi(\alpha)_{\scriptscriptstyle \mathcal{CR}} = \hat{\Psi}\circ \alpha_{\scriptscriptstyle \mathcal{CR}}\circ \hat{\Psi}^{-1}.$$
\end{thm}
\begin{proof}
For any $\alpha \in \aut(\sigma_{n})$, the map $\Psi(\alpha)_{\scriptscriptstyle \mathcal{CR}}^{-1}\hat{\Psi}\alpha_{\scriptscriptstyle \mathcal{CR}}\hat{\Psi}^{-1}$ is a homeomorphism of $\mathcal{CR}(\sigma_{n})$ which, by Theorem~\ref{thm:prop-global-verraum} acts by the identity on the finite systems. By Proposition~\ref{prop:CRlimitperiodic}, the finite systems are dense in $\mathcal{CR}(\sigma_{n})$, so $\Psi(\alpha)_{\scriptscriptstyle \mathcal{CR}}^{-1}\hat{\Psi}\alpha_{\scriptscriptstyle \mathcal{CR}}\hat{\Psi}^{-1}$ is equal to the identity on all of $\mathcal{CR}(\sigma_{n})$.
\end{proof}

We can now define the map
$$\mathcal{V} \colon \aut_{1}(\autinfn) \to \textrm{Homeo}(\mathcal{CR}(\sigma_{n}))$$
$$\mathcal{V}(\Phi) = \hat{\Phi}.$$
By Proposition~\ref{prop:verraum-injective}, the map $\mathcal{V}$ is a homomorphism.
\begin{thm}
There is an exact sequence
$$1 \longrightarrow \hat{\mathbb{Z}} \stackrel{\mathcal{N}}\longrightarrow \degonen \stackrel{\mathcal{V}}\longrightarrow \textnormal{Homeo}(\mathcal{CR}(\sigma_{n})).$$
\end{thm}
\begin{proof}
Suppose first that $\Psi$ is in the image of $\mathcal{N}$. Then $\hat{\Psi}$ maps every $\sigma_{n}$-periodic orbit to itself, and hence leaves invariant every finite subshift. Since $\mathcal{CR}(\sigma_{n})$ is the closure of the set of finite subsystems in $X_{n}$ and $\hat{\Psi}$ is continuous, this means $\hat{\Psi}$ acts by the identity on $\mathcal{CR}(\sigma_{n})$, so $\mathcal{V}(\Psi) = \textrm{id}$.

Suppose now that $\mathcal{V}(\Psi) = \textrm{id}$. Then $\hat{\Psi}$ acts by the identity on $\mathcal{CR}(\sigma_{n})$, so $\hat{\Psi}$ must map every $\sigma_{n}$-periodic orbit to itself. By Proposition~\ref{prop:shifting-verraum}, this implies $\Psi$ is in the image of $\mathcal{N}$.
\end{proof}

\subsection{The stabilized \verraum}

Our goal now is to extend the \verraum $\hat{\Psi}$ of a degree one automorphism $\Psi \colon \autinfn \to \autinfn$ to the stabilized space $\mathcal{CR}^{\infty}(\sigma_{n})$ of chain recurrent subshifts in $X_{n}$.

\begin{defn}
For $n \ge 2$, we define the stabilized space of subshifts by
$$\mathcal{S}^{\infty}(\sigma_{n}) = \{Y \subseteq X_{n} \mid Y \textrm{ is compact and } \sigma^{k}(Y) = Y \textrm{ for some } k \ge 1\}.$$

Similarly, we define the stabilized space of chain recurrent subshifts by
$$\mathcal{CR}^{\infty}(\sigma_{n}) = \bigcup_{k=1}^{\infty}\mathcal{CR}(\sigma_{n}^{k})$$
where the union is taken in the space $\mathcal{S}^{\infty}(\sigma_{n})$.
\end{defn}

Note that equivalently, we have
$$\mathcal{S}^{\infty}(\sigma_{n}) = \bigcup_{k=1}^{\infty} \mathcal{S}(\sigma_{n}^{k})$$
where the union is taken in the space $\mathcal{K}(X_{n})$. In summary, we have the containments
$$\mathcal{CR}^{\infty}(\sigma_{n}) \subseteq \mathcal{S}^{\infty}(\sigma_{n}) \subseteq \mathcal{K}(X_{n}).$$

The space $\mathcal{S}^{\infty}(\sigma_{n})$ is not closed in $\mathcal{K}(X_{n})$; for instance, the sequence of singleton sets $\{\linf(1^{k}0^{k}).(1^{k}0{^k})\rinf \}$ all belong to $\mathcal{S}^{\infty}(\sigma_{n})$, but the limit in $\mathcal{K}(X_{n})$ is the singleton set $\{{}^{\infty}0.1^{\infty}\}$ which is not in $\mathcal{S}^{\infty}(\sigma_{n})$. Similarly, $\mathcal{CR}^{\infty}$ is not closed in $\mathcal{S}^{\infty}$: for example, let $Y$ be the subshift obtained as the orbit closure of the point $^{\infty}0.1^{\infty}$ and for each $k$ let $L_{2k+1}(Y)$ be the set of words of length $2k+1$ that appear in $Y$. For each $w \in L_{2k+1}(Y)$ let $q_{w}$ be the periodic point obtained by periodizing $w$. Then for each $k$ the set $Q_{k} = \{q_{w} \mid w \in L_{2k+1}(Y)\}$ is a finite $\sigma_{2}^{2k+1}$-subshift, and hence belongs to $\mathcal{CR}^{2k+1}(\sigma_{2})$, but the limit of $Q_{k}$ in the Hausdorff metric is the subshift $Y$ which does not belong to $\mathcal{CR}^{\infty}(\sigma_{2})$.

For every $k \ge 1$, the group $\aut(\sigma_{n}^{k})$ acts on $\mathcal{S}(\sigma_{n}^{k})$ and leaves invariant $\mathcal{CR}(\sigma_{n}^{k})$. Moreover this action extends in the sense that for every $\ell \ge 1$, $\aut(\sigma_{n}^{k})$ acts on $\mathcal{S}(\sigma_{n}^{k\ell})$ as well, arising from the containment $\aut(\sigma_{n}^{k}) \subseteq \aut(\sigma_{n}^{k\ell})$. It follows there is a well-defined action of $\autinfn$ on $\mathcal{S}^{\infty}(\sigma_{n})$ which leaves invariant the subset $\mathcal{CR}^{\infty}(\sigma_{n})$.

Analogous to the previous section, given $\varphi \in \autinfn$ we denote the action of $\varphi$ on $\mathcal{CR}^{\infty}(\sigma_{n})$ by
$$\varphi_{\scriptscriptstyle \mathcal{CR}} \colon \mathcal{CR}^{\infty}(\sigma_{n}) \to \mathcal{CR}^{\infty}(\sigma_{n}).$$

The spaces $\sinfn$ and $\crinfn$ both have topologies inherited as subspaces of $\mathcal{K}(X_{n})$, which are induced by the Hausdorff metric. They also can each be equipped with the final topology coming from the filtrations

$$\mathcal{CR}^{\infty}(\sigma_{n}) = \bigcup_{k=1}^{\infty}\mathcal{CR}(\sigma_{n}^{k}), \quad \mathcal{S}^{\infty}(\sigma_{n}) = \bigcup_{k=1}^{\infty} \mathcal{S}(\sigma_{n}^{k}).$$
In this topology, a set $U$ in $\sinfn$ (respectively in $\crinfn$) is open if and only if $U \cap \mathcal{S}(\sigma_{n}^{k})$ is open for all $k \ge 1$. The final topology on both $\sinfn$ and $\crinfn$ is finer than the topology induced by the Hausdorff metric. For example, consider the sequence of points $p_{k} = \rinf(10^{k})^{\infty}$ with the origin in the middle of the zero string. Each $p_{k}$ is a fixed point of some power of $\sigma_{2}$, and belongs to the space of stabilized $\sigma_{2}$-subshifts. The set $A = \{p_{k}\}_{k=1}^{\infty}$ is closed in the final topology, since each level $\mathcal{S}(\sigma_{2}^{i})$ contains at most finitely many $p_{k}$'s. However $A$ is not closed in $\mathcal{K}(X_{2})$, since the sequence of points $p_{k}$ converges to the point of all zeros, which is not in $A$.

Note that the action of $\autinfn$ on $\sinfn$ and $\crinfn$ is continuous in both the Hausdorff metric and the final topology.

Suppose now that $\Psi \in \degonen$, so that $\Psi(\sigma_{n}) = \sigma_{n}^{\pm 1}$. Then $\Psi(\sigma_{n}^k) = \sigma_{n}^{\pm k}$ for all $k \ge 1$ and $\Psi(\aut(\sigma_{n}^{k})) = \aut(\sigma_{n}^{k})$. Then we have, for every $k \ge 1$, the $k$th \verraum homeomorphism
$$\hat{\Psi}_{k} \colon \mathcal{CR}(\sigma_{n}^{k}) \to \mathcal{CR}(\sigma_{n}^{k})$$
which satisfies for every $\varphi \in \aut(\sigma_{n}^{k})$
$$\Psi(\varphi)_{\smallcr} = \hat{\Psi}_{k} \circ \varphi_{\smallcr} \circ \hat{\Psi}_{k}^{-1}.$$
Combining all of this we obtain the following.
\begin{thm}\label{thm:stabilizedverraum}
Let $\Psi \in \degonen$. There is a stabilized \verraum map
$$\hat{\Psi}_{\infty} \colon \crinfn \to \crinfn$$
which is bijective and a homeomorphism with respect to the final topology. Furthermore, $\hat{\Psi}_{\infty}$ satisfies, for every $\alpha \in \autinfn$,
$$\Psi(\varphi)_{\smallcr} = \hat{\Psi}_{\infty} \circ \varphi{\smallcr} \circ \hat{\Psi}_{\infty}^{-1}.$$
\end{thm}

We can now prove Theorem~\ref{thm:maintheorem} from the introduction.
\begin{thm}[Theorem~\ref{thm:maintheorem} of the introduction]\label{thm:maintheorem2}
Fix $m,n \ge 2$ and suppose $\Psi \colon \autinf(\sigma_{m}) \to \autinf(\sigma_{n})$ is an isomorphism of groups. There is a homeomorphism $\hat{\Psi} \colon \mathcal{CR}^{\infty}(\sigma_{m}) \to \mathcal{CR}^{\infty}(\sigma_{n})$ with respect to the final topologies such that $\Psi$ is $\mathcal{CR}^{\infty}(\sigma_{m}), \mathcal{CR}^{\infty}(\sigma_{n})$-hyperspatial, i.e. for all $\varphi \in \autinf(\sigma_{m})$ we have
$$\Psi(\varphi)_{\scriptscriptstyle \mathcal{CR}} = \hat{\Psi}\circ \varphi_{\scriptscriptstyle \mathcal{CR}}\circ \hat{\Psi}^{-1}.$$
Moreover, the map $\hat{\Psi}$ takes $\sigma_{m}$-periodic points to $\sigma_{n}$-periodic points.
\end{thm}
\begin{proof}
Given such $\Psi$, it was shown in~\cite[Corollary 35, Lemma 33]{schmiedingLocalMathcalEntropy2022} that there exists nonzero $k,j \in \mathbb{Z}$ such that $m^{\abs{k}} = n^{\abs{j}}$ and $\Psi(\sigma_{m}^{k}) = \sigma_{n}^{j}$. Upon precomposing and postcomposing $\Psi$ with the spatial isomorphisms $\Xi_{m}$ and $\Xi_{n}$ respectively, if necessary, we may assume both $k$ and $j$ are positive. We may choose topological conjugacies
$$H_{1} \colon (X_{m},\sigma_{m}^{k}) \to (X_{m^{k}},\sigma_{m^{k}}),$$
$$H_{2} \colon (X_{n},\sigma_{n}^{j}) \to (X_{n^{j}},\sigma_{n^{j}})$$
which induce spatial isomorphisms
$$H_{1,*} \colon \aut^{\infty}(\sigma_{m}^{k}) \to \aut^{\infty}(\sigma_{m^{k}}),$$
$$H_{2,*} \colon  \aut^{\infty}(\sigma_{n}^{j}) \to \aut^{\infty}(\sigma_{n^{j}}).$$
It is straightforward to check that, as sets, $\mathcal{CR}^{\infty}(\sigma_{m}^{k}) = \mathcal{CR}^{\infty}(\sigma_{m})$ and $\mathcal{CR}^{\infty}(\sigma_{n}^{j}) = \mathcal{CR}^{\infty}(\sigma_{n})$. Furthermore, the final topologies on $\mathcal{CR}^{\infty}(\sigma_{m})$ coming from both the usual filtration
$$\mathcal{CR}^{\infty}(\sigma_{m}) = \bigcup_{j=1}^{\infty}\mathcal{CR}(\sigma_{m}^{j})$$
and the filtration
$$\mathcal{CR}^{\infty}(\sigma_{m}) = \bigcup_{j=1}^{\infty}\mathcal{CR}(\sigma_{m}^{kj})$$
are the same. Indeed, if $U \cap \mathcal{CR}(\sigma_{m}^{j})$ is open for all $j$, then clearly $U \cap \mathcal{CR}(\sigma_{m}^{kj})$ is open for all $j$, and for any $\ell$, if $U \cap \mathcal{CR}(\sigma_{m}^{k\ell})$ is open then $U \cap \mathcal{CR}(\sigma_{m}^{\ell})$ is open since $\mathcal{CR}(\sigma_{m}^{\ell})$ is a subspace of $\mathcal{CR}(\sigma_{m}^{k\ell})$. Likewise for $\mathcal{CR}^{\infty}(\sigma_{n})$.

It follows that the $H_{i}$'s induce homeomorphisms
$$H_{1,\mathcal{CR}} \colon \mathcal{CR}^{\infty}(\sigma_{m}) \to \mathcal{CR}^{\infty}(\sigma_{m^{k}})$$
$$H_{2,\mathcal{CR}} \colon \mathcal{CR}^{\infty}(\sigma_{n}) \to \mathcal{CR}^{\infty}(\sigma_{n^{j}})$$
with respect to the final topologies.

It follows from the definitions that $\aut^{\infty}(\sigma_{m}) = \aut^{\infty}(\sigma_{m}^{k})$ and $\autinfn = \aut^{\infty}(\sigma_{n}^{j})$. Then we have the composition of isomorphisms
$$\aut^{\infty}(\sigma_{m^{k}}) \stackrel{H_{1,*}^{-1}}{\longrightarrow} \aut^{\infty}(\sigma_{m}^{k})$$
$$= \aut^{\infty}(\sigma_{m}) \stackrel{\Psi}\longrightarrow \autinfn = \aut^{\infty}(\sigma_{n}^{j}) \stackrel{H_{2,*}}{\longrightarrow} \aut^{\infty}(\sigma_{n^{j}})$$
$$ = \aut^{\infty}(\sigma_{m^{k}}).$$
Note that $H_{2,*}(\sigma_{n}^{j}) = \sigma_{n^{j}}$, $H_{1,*}^{-1}(\sigma_{m^{k}}) = \sigma_{m}^{k}$, and $\Psi(\sigma_{m}^{k}) = \sigma_{n}^{j}$. Taken with all of the above, since both $H_{2,*}$ and $H_{1,*}^{-1}$ are spatial via the maps $H_{1},H_{2}$ which induce the homeomorphisms $H_{i,\mathcal{CR}}$ on the respective $\mathcal{CR}^{\infty}$'s with respect to the final topologies, we have reduced to the case where we have a degree one isomorphism
$$\Psi \colon \aut^{\infty}(\sigma_{m^{k}}) \to \aut^{\infty}(\sigma_{m^{k}}).$$
Then by Theorem~\ref{thm:stabilizedverraum} we have the \verraum map
$$\hat{\Psi}_{\infty} \colon \mathcal{CR}^{\infty}(\sigma_{m^{k}}) \to \mathcal{CR}^{\infty}(\sigma_{m^{k}})$$
which is a homeomorphism with respect to the final topologies, and satisfies
$$\Psi(\alpha)_{\smallcr} = \hat{\Psi}_{\infty} \circ \alpha_{\smallcr} \circ \hat{\Psi}_{\infty}^{-1}$$
for every $\alpha \in \autinf(\sigma_{m^{k}})$.
Finally, by construction we have that $\hat{\Psi}$ takes $\sigma_{m^{k}}$-periodic points to $\sigma_{m^{k}}$-periodic points.
\end{proof}

Let us remark on the \verraum and the question posed in the introduction. The set of periodic points, with the metric inherited from $X_{n}$, embeds isometrically into $\mathcal{S}^{\infty}(\sigma_{n})$ with the Hausdorff metric. It is clear that the image of this embedding is contained in the set of stabilized chain recurrent systems $\mathcal{CR}^{\infty}(\sigma_{n})$. If the \verraum $\hat{\Psi}$ of $\Psi$ was continuous on $\mathcal{CR}^{\infty}(\sigma_{n})$ with the Hausdorff metric, it would follow that the \verraum $\hat{\Psi}$ is continuous on the set of periodic points in $X_{n}$ in the usual topology on $X_{n}$. If this continuity was uniform, density of the periodic points in $X_{n}$ would allow $\hat{\Psi}$ to extend to a self homeomorphism of $X_{n}$, giving spatiality for $\Psi$ at the level of $X_{n}$.

\begin{rem}
There is a partial ordering on $\mathcal{S}(\sigma_{n})$ and $\mathcal{S}^{\infty}(\sigma_{n})$, and hence on $\mathcal{CR}(\sigma_{n})$ and $\mathcal{CR}^{\infty}(\sigma_{n})$, given by containment. Given a degree one automorphism $\Psi \in \aut_{1}(\autinfn)$, the \verraum $\hat{\Psi}$ and its stabilized version $\hat{\Psi}_{\infty}$ respect this partial ordering, in the sense that if $Y \subseteq Z$, then $\hat{\Psi}(Y) \subseteq \hat{\Psi}(Z)$, and likewise for $\hat{\Psi}_{\infty}$.
\end{rem}

Recall that the $\sigma_n$-periodic points embed isometrically
into $\mathcal{S}^{\infty}(\sigma_n)$.
The following example shows that there exists
continuous map $\mathcal{S}^{\infty}(\sigma_{n})
\to \mathcal{S}^{\infty}(\sigma_{n})$
which maps this
isometric copy of $\Per_{\sigma_n}$
to itself and whose restriction
to $\Per_{\sigma_n}$ commutes with $\sigma_n$
but is not continuous.
and induces a continuous self-map of $\mathcal{S}^{\infty}(\sigma_n)$.
This shows that in this setting $\mathcal{S}^{\infty}(\sigma_{n})$-hyperspatiality is
strictly weaker then spatiality.

\begin{exam}
    Let $y$ be a totally transitive point in $X_n$,
    i.e. a point whose orbit under $\sigma_n^k$ is dense in $X_n$
    for every $k \in \bbN$.
    Let $Z$ be a proper subshift of $(X_n,\sigma_n^k)$ and
    let $w$ be a word over the alphabet $\{1,\dots,n\}^k$
    of length $m$ forbidden in $Z$.
    For each $i \in \{0,\dots,m-1\}$ there is $\ell_i \in \bbZ$
    with $\ell_i = i \mod k$
    such that $y_{[\ell_i, \ell_i+km-1]}=w$.
    Take $\ell:=\max \{\abs{\ell_i}+km, i=0,\dots,k-1\}$.
    Every element of $X_n$
    which contains $y_{[-\ell,\ell]}$
    also contains $w$ at a position which is a multiple of $k$.
    Thus the SFT $Y^{(\ell)}$, whose
    only forbidden word is $y_{[-\ell,\ell]}$,
    contains $Z$.

    All in all we have constructed a strictly increasing sequence of SFTs $Y^{(\ell)}$ such that every proper  stabilized subshifts is eventually contained in $Y^{(\ell)}$.
    Now we construct a sequence of automorphisms $\varphi_k: (Y^{(k)} ,\sigma_n) \to (Y^{(k)},\sigma_n)$
    such that $\varphi_k(x)=\varphi_\ell(x)$ for
    every $\ell < k$, $x \in Y^{(\ell)}$,
    and such that the minimal neighborhood
    of $\varphi_k$ goes to infinity.
    We can start by simply taking $\varphi_1$ as the identity on $Y^{(1)}$.
    To construct $\varphi_k$ based on $\varphi_{k-1}$
    let $Z^{(k)}$ be a mixing SFT, contained in $Y^{(k)}$
    and disjoint from $Y^{(k-1)}$.
    Let $\alpha_k$ be an inert automorphism on $Z^{(k)}$
    whose minimal neighborhood has at least size $k$.
    Produce an inert automorphism of $Y^{(k-1)} \cup Z^{(k)}$
    by combining $\varphi_{k-1}$ with $\alpha_{k}$.
    By the Inert Extension Theorem
    \cite[Theorem 2.1]{kimStructureInertAutomorphisms1991}
    we can extend this automorphism to an inert
    on $Y^{(k)}$.

    We can combine all these maps
    into a map
    $\zeta: \bigcup_{k \in \bbN} Y^{(k)} \to \bigcup_{k \in \bbN} Y^{(k)}$
    by setting $\zeta(x) = \varphi_k(x)$ for $x \in Y^{(k)}$.
    The compatibility of the $\varphi_k$ ensures that
    this map is well-defined.
    Notice that $\zeta$ is continuous when restricted to $Y^{(k)}$
    but it is not a continuous map on $\bigcup Y^{(k)} \cap \Per(\sigma_n) \subseteq X_n$
    endowed with the subspace topology,
    since the size of the minimal neighborhood of the $\varphi_k$ goes to infinity.
    Since $\zeta$ commutes with $\sigma_n$
    and $\Per(\sigma_n) \subseteq \bigcup_{k \in \bbN} Y^{(k)}$,
    our map $\zeta$ induces a bijective self-map
    on $\Per(\sigma_n)$.

    Similarly we can define a bijective self-map $\tilde{\zeta}$ of $S^\infty(\sigma_n)$
    given by $\tilde{\zeta}(Q):=\varphi_k(Q)$ for $Q \subseteq Y^{(k)}$
    and $\tilde{\zeta}(X_n)=X_n$.
    We finish by showing that this map $\tilde{\zeta}$
    is a homeomorphism.
    First of all we only have to show continuity of $\tilde{\zeta}$,
    the same argument then applies to its inverse.
    Second, by the nature of the final topology,
    we only have to show that the restriction of
    $\tilde{\zeta}$ to $\mathcal{S}(\sigma_n^k)$
    for every $k \in \bbN$
    is continuous. By the usual argument of passing
    to the alphabet of size $n^k$, we only have to deal with the case
    $k=1$.
    Notice that $\tilde{\zeta}$ leaves $\mathcal{S}(\sigma_n)$
    invariant.
    Now let $(W_k)_{k \in \bbN}$ be a sequence of
    subshifts in $\mathcal{S}(\sigma_n)$
    converging to some subshift $W$.

    First we deal with the case that
    $W \neq X_n$.
    Then there is $\ell \in \bbN$
    such that $W \subseteq Y^{(\ell)}$.
    In other words, $y_{[-\ell,\ell]}$
    is not in the language of $W$
    and therefore also not in the language of $W_k$
    for sufficiently large $k$.
    Hence $W$ and $W_k$
    both lie eventually in $Y^{(\ell)}$.
    But $\tilde{\zeta}$ is continuous on
    $\mathcal{S}(\sigma_{Y^{(\ell)}})$
    since there it is induced by the automorphism
    $\varphi_\ell$. Hence $\tilde{\zeta}(W_k) \to \tilde{\zeta}(W)$.

    Now assume that $W=X_n$.
    Then for every $\ell \in \bbN$, $W_k$ eventually contains
    every word of length $2\ell+1$.
    Hence $W_k \not\subseteq Y^{(\ell)}$.
    Since $\tilde{\zeta}$ leaves $Y^{(\ell)}$
    invariant, we also get
    $\tilde{\zeta}(W_k) \not\subseteq Y^{(\ell)}$
    for all $\ell$ and for sufficiently large $k$
    depending on $\ell$.
    We claim that this implies $\tilde{\zeta}(W_k) \to X_n=\tilde{\zeta}(W)$.
    Assume otherwise, then by compactness there would be
    another accumulation point $V$ of
    $(\tilde{\zeta}(W_k))_{k \in \bbN}$.
    But then $V$ would be contained
    in some $Y^{(\ell)}$
    and hence, by the argument from the first case, there would be an infinite
    set of indices $k$ for which
    $W_k \subseteq \tilde{\zeta}^{-1}(Y^{(\ell)}) =Y^{(\ell)}$,
    contradiction.

    Both cases taken together imply that $\tilde{\zeta}$
    is continuous.
\end{exam}

\section{Some consequences}
Here we record some consequences of the \verraum construction.
\subsection{Being free is characteristic}

While we only obtain continuity of the Verräumlichung at the level of the space of subshifts, we still get the following somewhat surprising corollary.

\begin{cor}[Freeness of finite order elements is a characteristic property]\label{cor:freeness}
Let $\Psi \in \autautinfn$.
Let $\varphi \in\autinfn$ be
an element of order $k$ for which all orbits have size $k$.
Then $\Psi(\varphi)$ has the same property.
\end{cor}
\begin{proof}
    It is enough to show this for the action of $\Psi(\varphi)$ on $\sigma_n$-periodic points. The result then follows from \Cref{thm:prop-global-verraum}.
\end{proof}

\subsection{Countability of the non-profinite part of $\asan$}
Recall we have an exact sequence
$$1 \longrightarrow \hat{\mathbb{Z}} \stackrel{\mathcal{N}}\longrightarrow \degonen \stackrel{\mathcal{V}}\longrightarrow \textnormal{Homeo}(\mathcal{CR}(\sigma_{n})).$$

Based on ideas from~\cite{saloConjugacyReversibleCellular2022}, we will prove the following.
\begin{thm}
  \label{thm:left-cosets-countable}
The number of left cosets of $\mathcal{N}(\hat{\bbZ})$ in $\aut(\autinfn)$ is countable.
\end{thm}
Let $n \geq 5$. Let $\Orb_{k}(\sigma_n^\ell)$ be the set of orbits of
$\sigma_n^\ell$ of length $k$.
Let $G_m(\sigma^\ell_n)$ be the subgroup of $\Aut(\sigma_n^{m\ell})$ generated by
$\sigma_n^\ell$ and $\Simp^{(m\ell)}(\Gamma_n)$. In particular, $G_m(\sigma_n^\ell)$ is
generated by $3$ elements.

 We think of elements of $(X_n,\sigma_n^{2})$ as organized
  in two tracks, both containing symbols from the alphabet $\{1,\dots,n\}$. We denote the top track of $x \in (X_n,\sigma_n^2)$ by $x^\topt \in (X_n, \sigma_n)$
  and the bottom track by $x^\bott$.
  Let $\pi$ be a permutation of $\{1,\dots,n\}$
  and let $w$ be a word over the alphabet $\{1,\dots,n\}$. Let
  $g_{w,\pi}$ be the automorphism of $(X_n,\sigma_n^{2})$ which
  applies $\pi$ at position $i$ at the bottom track if
  the top track contains $w$ starting at position $i$.
  Clearly $G_{2}(\sigma_n)$ contains $g_{b,\pi}$ for every permutation $\pi$
  and symbol $b$.
  Now let $w_1, w_2$ be words over the alphabet $\{1,\dots,n\}$  respectively. Let $\pi_1,\pi_2$ be two permutations
  of $\{1,\dots,n\}$.
  Let $\gamma$ be the automorphism which shifts only the top track to the left.
  Then $\gamma^{-k} \circ g_{w_2,\pi_2}\circ \gamma^{k}$
  is the automorphism that applies $\pi_2$ in the bottom track
  at position $i$ if the first track contains $w_2$ starting at
  position $i+k$.
  Therefore $[g_{w_1,\pi_1}, \gamma^{-|w_1|}\circ g_{w_2,\pi_2}\circ
  \gamma^{|w_1|}]$
  equals $g_{w_1w_2,[\pi_1,\pi_2]}$.
  This shows that $G_2(\sigma_n)$ contains $g_{w,\pi}$ for every
  word $w$ and every even permutation $\pi$ since $\Alt(\{1,\dots,n\})$
  is perfect.

  \begin{lem}

  Let $n\geq 5$ and $x \in \Per_k(\sigma_n^2)$ be a point of minimal $\sigma_n^2$-period $k$.
  There is $g \in G_2(\sigma_n)$ such that
  the top track of $g(x)$ has minimal period $k$.
\end{lem}
\begin{proof}
  Let $y$ be a point in the $G_2(\sigma_n)$ orbit of $x$
  for which $m:=\per(y^\topt)$ is maximal. If $m=k$ we are done.
  So assume $m<k$.
  Now we use automorphisms of the
  type constructed above to make the first
  $m$ symbols of the bottom track all equal to one.
  More precisely, for $a \in \{1,\dots,n\}$ let $\pi_a$ be an even permutation
  on $\{1,\dots,n\}$ such that $\pi_a(a)=1$.
  Set $h := \prod_{j=0}^{m-1} g_{w_j,\theta_j}$ with $w_j:=y^\topt_{[j,j+k-1]}$ and $\theta:=\pi_{y^\bott_j}$.
  The automorphisms in the product commute with each other
  and the symbols at position $0,\dots,m-1$ in
  the bottom track of $h(y)$ are all equal to $1$.
  Hence the bottom track of $h(y)$ either is constant
  or has minimal period longer then $m$.
  The top track of $h(y)$ on the other hand equals $y^\topt$.
  The bottom track of $h(y)$ can't be constant, because then
  $h(y)$ would have minimal $\sigma_n^2$-period $m \neq k$.
  Hence it must have minimal period $\ell>m$.
  Let $z \in G_2(\sigma_n)(x)$ be the configuration
  obtained from $h(y)$ by swapping the top and bottom track.
  Then $z$ is a configuration in the $G_2(\sigma_n)$-orbit
  of $x$ whose top track has minimal period longer than $m$.
  But this contradicts the maximality of $m$.
\end{proof}

\begin{prop}
\label{prop:centralizer-trivial-orbit-g2-1}
  Let $k \in \bbN$ and let $n\geq 5, \ell \ge 1$.
  The centralizer of the image of $G_2(\sigma_n^\ell)$ in $\Sym(\Orb_{k}(\sigma_n^{2\ell}))$ is trivial.
\end{prop}
\begin{proof}
  The topological conjugacy between
  $(X_n,\sigma_n^\ell)$ and $(X_{n^\ell}, \sigma_{n^\ell})$
  allows us to only deal with the case $\ell=1$.
  Let $h$ be a non-trivial element of $\Sym(\Orb_{k}(\sigma_n^{2}))$.
  For a point $x \in \MPer_k(\sigma_n^2)$
  denote by $[x]$ its $\sigma_n^2$-orbit.
  Take two elements $x,y$ of $\MPer_k(\sigma_n^{2})$ belonging to
  different $\sigma_n^{2}$ orbits such
  that $h([x])=[y]$.
  By the previous lemma there is
  $f \in G_2(\sigma_n)$ such that
  the top track of $f(x)$
  has minimal period $k$.
  We now want to find $g \in G_2(\sigma_n)$ which
  fixes $f(y)$ and moves $f(x)$ out of its $\sigma_n$-orbit.
  To do so let $\pi_{a,b}$ for $a,b \in \{1,\dots,n\}, a\neq b$
  be an even permutation which fixes $a$ but not $b$.

  If $f(x)^\topt$ and $f(y)^\topt$ are
  not in the same $\sigma_n$-orbit,
  set $g:=g_{w,\theta}$ with \[w:=f(x)^\topt_{[0,\dots,k-1]}
  \text{ and }\theta=\pi_{f(y)^\bott_0,f(x)^\bott_0}.\]

  If $f(x)^\topt$ and $f(y)^\topt$ are in the same $\sigma_n$-orbit, then there are $m,j \in \{0,\dots,k-1\}$
  such that $f(x)^\topt = \sigma_n^j(f(y)^\topt)$
  and $f(x)^\bott_m \neq \sigma_n^j(f(y)^\bott)_m$.
  We can now simply set $g:=g_{w,\theta}$ with
  \[w:=f(x)^\topt_{[m,\dots,m+k-1]} \text{ and }\theta=\pi_{f(y)^\bott_{m+j},f(x)^\bott_m}.\]
  In both cases $g$ changes precisely one symbol
  in the bottom track of $f(x)$ and no symbol in $f(y)$.
  Since $g$ does not change the top tracks of $f(x)$
  and $f(y)$, we have $g(f([x])) \neq f([x])$ but
  $g(f([y]))=f([y])$.
  This shows that $h$ can not be in the
  centralizer of the action $G_2(\sigma_n)$ on the $\sigma_n^2$-orbits.
  Otherwise \[h(g(f([x])))=g(f(h([x])))=g(f([y]))=f([y])=f(h([x]))=h(f([x])).\]
  But this would contradict $g(f([x]))\neq f([x])$.
\end{proof}
We are now ready to prove the countability result.

\begin{proof}[Proof of Theorem~\ref{thm:left-cosets-countable}]
  It is enough to show that $\mathcal{N}(\hat{\bbZ})$
  has countably many left cosets in
  $\aut_\ell(\autinf(\sigma_n))$
  for $\ell\geq 3$.

  We want to show that every pair of automorphisms
  in
  $\aut_\ell(\autinf(\sigma_n))$ which agree on
  $G_2(\sigma_n^\ell)$ must lie in the same $\profact(\hat{\bbZ})$
  coset.
  For this it is enough to show
  that every automorphism $\Psi \in \aut_\ell(\autinf(\sigma_n))$
  which fixes all elements of $G_2(\sigma_n^\ell)$
  is in $\mathcal{N}(\profint)$.
  Let $\Psi$ be such an automorphism. For all $\varphi \in G_2(\sigma_n^\ell)$ and
  $x \in \Per(\sigma_n^{2\ell})$ we know that
  \begin{align*}
    \varphi(x) = \Psi(\varphi)(x)= ((\spatial)^{-1} \circ \varphi \circ \spatial)(x).
  \end{align*}
  Hence $\spatial^{(2\ell)}$ is in the centralizer of $G_2(\sigma_n^\ell)$ in
  $\Sym(\Orb_{k}(\sigma_n^{2\ell}))$ for $k:=\per(x)$.
  But this centralizer is trivial by \Cref{prop:centralizer-trivial-orbit-g2-1}.
  Therefore for every
  $x \in \Per(\sigma_n^{2\ell})$ there is $k(x) \in \bbZ$ such that
  $\spatial(x) = \sigma_n^{2k(x)}(x)$. Then by Proposition~\ref{prop:shifting-verraum} we know that $\Psi$ is
  in the image of $\mathcal{N}(\hat{\bbZ})$ in $\Aut_\ell(\Aut^\infty(\sigma_n^{2\ell}))$.
\end{proof}

\bibliographystyle{plain}
\bibliography{references}
\end{document}